\documentclass[12pt,reqno,b5paper]{amsart}%
\usepackage{anysize}
\marginsize{13mm}{13mm}{10mm}{10mm}
\usepackage[noadjust]{cite}
\usepackage[T1]{fontenc}
\usepackage{txfonts,amsmath,esint}
\usepackage[colorlinks,urlcolor=black,citecolor=blue,linkcolor=blue]{hyperref}
\usepackage{amsmath}
\usepackage{amsfonts}
\usepackage{amssymb}
\usepackage{graphicx}

\DeclareSymbolFont{CM}{OMX}{cmex}{m}{n}
\DeclareMathSymbol{\sumop}{\mathop}{CM}{"50}
\renewcommand{\sum}{\sumop}

\allowdisplaybreaks[4]

\let\oldsection\section
\renewcommand\section{\setcounter{equation}{0}\oldsection}

\catcode`!=11
\let\!int\int \def\int{\displaystyle\!int}
\let\!lim\lim \def\lim{\displaystyle\!lim}
\let\!sum\sum \def\sum{\displaystyle\!sum}
\let\!sup\sup \def\sup{\displaystyle\!sup}
\let\!inf\inf \def\inf{\displaystyle\!inf}
\let\!max\max \def\max{\displaystyle\!max}
\let\!min\min \def\min{\displaystyle\!min}

\catcode`!=12

\allowdisplaybreaks
\newcommand{\RNum}[1]{\uppercase\expandafter{\romannumeral #1\relax}}

\newtheorem{thm}{Theorem}[section]
\newtheorem{prop}{Proposition}[section]
\newtheorem{lem}{Lemma}[section]

\newtheorem{rem}{Remark}[section]
\def\pf{\it{Proof.}\rm\quad}
\newcommand{\be}{\begin{equation}}
\newcommand{\ee}{\end{equation}}
\newcommand{\bes}{\begin{eqnarray*}}
\newcommand{\ees}{\end{eqnarray*}}
\newcommand{\md}{\mathrm{d}}
\newcommand{\na}{\nabla}

\newcommand{\dg}{\mathrm{div}}
\newcommand{\cl}{\mathrm{curl}}
\def\R{\mathbb{R}}

\newcommand{\bea}{\begin{aligned}}
\newcommand{\eea}{\end{aligned}}
\newcommand{\br}{\bar{\rho}}
\newcommand{\bx}{\bar{\xi}}
\allowdisplaybreaks

\begin{document}
\title[Global existence of large strong solutions to the micropolar equations]
{Global existence of large strong solutions to the 3D compressible micropolar equations with density-dependent viscosities}\thanks{Xu's research was supported by the NSFC (12371238)}.
\author{Tiantian Liu$^1$,\quad Xinying Xu$^{2}$}
\thanks{E-mails:liutiantian\_309@163.com (T. Liu); Corresponding author, xinyingxu@xmu.edu.cn (X. Xu)}
\dedicatory{1.School of Mathematical Sciences, Xiamen University\\ Xiamen 361005, P. R. China\\
2. School of Mathematical Sciences, Xiamen University\\ Xiamen 361005, P. R. China}
\vspace{-5mm}
\begin{abstract}
In this paper, we consider the Cauchy problem for the three-dimensional (3D) compressible micropolar equations with density-dependent viscosities. By classical $L^p$ estimates to elliptic system, defining suitable energy functions, after the elaborate index analysis, we establish the global existence of strong solutions to the system with all viscosities depending density as long as the initial density is large enough. By product, We also prove the existence of global large strong solutions with only large initial shear viscosity $\mu_1(\rho)$ and angular viscosity $\mu_2(\rho)$, which consists with the research work \cite{ZZ2014} of 3D compressible Navier-Stokes equations with constant viscosities and improve the recent result of \cite{HLZ} in some extent where they requiring both initial shear viscosity and bulk viscosity to be large.

\end{abstract}
\maketitle
\section{Introduction}
Micro fluid model introduced by Eringen in the 1960s deals with a class of fluids which exhibit certain microscopic effects arising from the local structure and micro-motions of the fluid elements. A subclass of these fluids is the micropolar fluids which exhibit the micro-rotational effects and micro-rotational inertia and may represent fluids that consist of rigid, randomly oriented (or spherical particles) suspended in a viscous medium in physics, where the deformation of fluid particles is ignored. Micropolar fluid equations can describe many phenomena that appear in a large number of certain complex anisotropic fluids such as the suspensions, and animal blood, which cannot be characterized appropriately by the Navier-Stokes system. The detailed derivation of the micropolar fluid equations from the general constitutive laws, together with an extensive review of the mathematical theory and the applications of this particular model, can be founded in the monographs \cite{Eringen1966, Lg1999, KLL2019, KL2020} and references therein.

In this paper, we consider the  three-dimensional compressible micropolar fluids with density-dependent viscosities introduced in \cite{Eringen1966} as follows:
\be\label{1.1}
\left\{
\begin{aligned}
&\rho_t + \dg(\rho u) = 0,\\
&(\rho u)_t+\dg(\rho u \otimes u)+\nabla P=\dg\left(2(\mu_1(\rho)+ \xi(\rho)) d(u)\right)\\
&\qquad\qquad\qquad+\nabla\left(\lambda_1(\rho)\dg u\right)-\cl(\xi(\rho)\cl u)+2\cl\left(\xi(\rho) w\right),
\\
&(\rho w)_t+\dg(\rho u \otimes w)+4\xi(\rho)w=\dg\left((\mu_2(\rho)\nabla w\right)\\
&\qquad\qquad\qquad\qquad\qquad+\nabla \left(\mu_2(\rho)+\lambda_2(\rho)) \dg w\right)+2\xi(\rho)\cl u,
\end{aligned}
\right.
\ee
where $t\ge 0$ is time, $x\in \R^3$ is the spatial coordinate, and unknown functions $\rho=\rho(x,t),$ $u=(u^1,u^2,u^3)(x,t),$ $w=(w^1,w^2,w^3)(x,t)$  are the fluid density, velocity, micro-rotational and velocity, respectively. The pressure $P$ is given by
\begin{equation}\label{1.2}
P=A\rho^\gamma,\,\,A>0,\,\,\gamma>1,
\end{equation}
and the deformation tensor is defined by
\begin{equation}\label{d}
 d (u) \triangleq \frac{1}{2} \left[\nabla u + (\nabla u)^\top \right].
\end{equation}
The constants $\mu_1(\rho))=\mu_1\rho^\alpha$ and $\lambda_1(\rho)=\lambda_1\rho^\alpha$ are the shear and bulk viscosity coefficients of the flow and satisfy the physical restrictions
\begin{equation}\label{1.3}
\mu_1>0,\,\,2\mu_1+3\lambda_1\geq 0,
\end{equation}
$\xi(\rho)=\bar{\xi}\rho^\beta\,(\bar{\xi}>0)$ is the dynamics microrotation viscosity, $\mu_2(\rho))=\mu_2\rho^\alpha$ and $\lambda_2(\rho)=\lambda_2\rho^\alpha$
are the angular viscosities satisfying
\begin{equation}\label{1.4}
\mu_2>0,\,\,2\mu_2+3\lambda_2\geq 0,
\end{equation}
where $\alpha>0$ and $\beta\geq 0$.

We consider the Cauchy problem of (\ref{1.1})--(\ref{1.4}) with the following initial data and far-field condition:
\begin{equation}\label{1.5}\begin{aligned}
&(\rho,\,u,\,w)(x,0)=(\rho_0,\,u_0,\, w_0)(x),\quad x\in \R^3,\\
&(\rho-\bar{\rho},\,u,\,w)(t,x)\rightarrow (0,0,0),\quad as\,\, |x|\rightarrow\infty,\,\,t\geq 0.
\end{aligned}
\end{equation}

When all the viscosities including shear, bulk viscosities, dynamics microrotation viscosity and angular viscosities are positive constants $(\alpha=0,\,\beta=0)$ and allows vacuum,  Chen et al. \cite{CXZ2015} obtained the global existence of week solution to the Cauchy problem for the 3D compressible micropolar equations with small initial data.

When there is no microstructure ($w=0$ and $\bx=0$), the micropolar equations reduces to Navier-Stokes equations. For the constant viscosities case, the global strong solutions is obtained if the viscosity coefficient is large enough in \cite{DZZ2012}. For the density-dependent viscosities case, Yu \cite{Yu2025} and Huang et al. \cite{HLZ} proved the global existence of strong solutions as long as the initial density is large enough, Xin and Zhu \cite{XZ2021} obtained the large global solution of Navier-Stokes equations under certain assumptions on $\alpha,\,\beta$ and $\gamma$. We will review some well-posedness results of 3D micropolar equations when viscosities depend on the density. For the micropolar system with density-dependent viscosities, Zhong \cite{Z2023} proved the global existence and exponential stability of strong solutions of Cauchy problem to 3D compressible micropolar equations, Si et al. \cite{SXZ2025} obtained the global well--posedness of solutions to the 3D initial boundary value problem of  micropolar-MHD equations when the viscosities coefficient are $C^1$ functions of the density under small initial data. It is obvious the results of \cite{SXZ2025} are also true to the 3D Cauchy problem of micropolar equations when the viscosities coefficient are $C^1$ functions of the density under small initial data. For the coefficients are power functions of the density, Lu and Qiao \cite{LQ2025} established the global existence of strong solutions to the initial boundary value problem of 3D inhomogeneous incompressible micropolar equations with large initial density.

Without lose of generality, we explain some notations and conventions, which will be used frequently in this paper. Firstly, denote
\begin{equation*}
\int f(x) \md x\triangleq \int_{\R^3} f(x) \md x.
\end{equation*}
Next, for $1\leq p\leq\infty$ and $k\in Z^+$, the following standard homogenous Sobolev spaces are defined as follows:
\begin{equation*}
\begin{cases}
L^p\triangleq L^p(\R^3),\quad W^{k,p}\triangleq W^{k,p}(\R^3), \quad H^k \triangleq W^{k,2},\\[2mm]
D^{k,p}\triangleq D^{k,p}(\R^3) \triangleq\left\{u \in L_{\rm loc}^1~|~\nabla^k{u} \in L^p\right\},\quad D^k\triangleq D^{k,2}.
\end{cases}
\end{equation*}

Our main result can be stated as follows:
\begin{thm}\label{th1}
Let $\Omega=\mathbb{R}^3$ and $\br$ be a constant. Suppose that the initial data $(\rho_0,\,u_0,\,w_0)$ satisfies
\begin{equation}\label{1.6}
\frac{3}{4}\br\leq \rho_0\leq\frac{5}{4}\br,\,\, \rho_0-\br\in D^{1,2}\cap D^{1,q},\quad 3<q<6,\,\, G(\rho_0)\in L^1,\,\, u_0\in H^2,\,\, w_0\in H^2.
\end{equation}
Let $\alpha,\,\beta$ and $\gamma$ satisfy
\begin{equation}\label{1.7}
\alpha>1,\quad 0\leq\beta< \alpha,\quad \gamma>1,
\end{equation}
and
\begin{equation}\label{1.8}
\alpha>\max \left\{\frac{\gamma+1}{2},\,\gamma-1\right\}.
\end{equation}
Then there exists a positive constant
$$\Lambda_0=\Lambda_0\left(A,\,\mu_1,\,\lambda_1,\,\mu_2,\,\lambda_2,\,\bx,\,\alpha,\,\beta,\,\gamma,\,\|G(\rho_0)\|_{L^1},\,\|\na \rho_0\|_{L^2},\,\|\na \rho_0\|_{L^q},\,\|u_0\|_{H^2},\,\|w_0\|_{H^2}\right)$$
such that if
\begin{equation}\label{1.9}
\br\geq \Lambda_0,
\end{equation}
then the problem (\ref{1.1})--(\ref{1.5}) admits a unique global solution $(\rho,u,w)$ in $\mathbb{R}^3\times(0,\infty)$ satisfying
\begin{equation}\label{1.10}
\frac{2}{3}\br\leq \rho_0\leq\frac{4}{3}\br
\end{equation}
and
\be\left\{
\begin{aligned}
&\rho\in C([0,\infty);D^{1,2}\cap D^{1,q}),\quad \rho_t\in C([0,\infty);L^q),\\
&\na u\in C([0,\infty);H^1)\cap L^2((0,\infty);W^{1,q}),\\
&u_t\in L^\infty((0,\infty);L^2)\cap L^2((0,\infty);H^1),\\
& w\in C([0,\infty);H^2)\cap L^2((0,\infty);W^{2,q}),\\
&w_t\in L^\infty((0,\infty);L^2)\cap L^2((0,\infty);H^1).\label{1.11}
\end{aligned}\right.
\ee
\end{thm}

\begin{rem}
If $\bx=0$, the system (\ref{1.1}) will reduce to Navier-Stokes system with density-dependent viscosities. This result is consistent with the results obtained by \cite{HLZ} to the Navier-Stokes equations. To the best of our knowledge, this work presents the first global existence result for large strong solutions to the 3D micropolar equations without small initial data assumption, a requirement that has been essential in all existing results \cite{CXZ2015,SXZ2025,Z2023}.
\end{rem}

\begin{rem}
In addition to the same $\alpha,\,\beta,\,\gamma$ requirements of the Navier-Stokes equations in \cite{HLZ}, we also require $0\leq\beta<\alpha$ in (\ref{1.7}), which play a crucial role to the micropolar model, see details in  (\ref{3.17}), (\ref{3.18}), (\ref{3.31}), (\ref{3.57}) and (\ref{3.64}).
\end{rem}

\begin{rem}
Compared with the assumption $\alpha>1, 0\leq \beta \leq \frac{\alpha+1}{2}$ of inhomogeneous incompressible micropolar equations with density-dependent viscosities in \cite{LQ2025}, the assumption (\ref{1.7}) of Theorem \ref{th1} implies that $\beta$ has a wider range in our research work due to the fact $\frac{\alpha+1}{2}<\alpha$.
\end{rem}

\begin{rem}
The form of the micropolar equations with density-dependent viscosities differs fundamentally from that of the constant-viscosity model. The equation (4.16) of \cite{Eringen1966} implies that the momentum equation of micropolar system should have the term $\cl\cl u$, see detail of $(\ref{1.1})_2$. The micropolar equations of \cite{LQ2025} may be incorrect where they missed the rotation term.
\end{rem}

\begin{rem}\label{re5}
We can prove that Theorem \ref{th1} also holds for the 3D compressible system (\ref{1.1}) with $\lambda_1(\rho)$ and $\lambda_2(\rho)$ are nonnegative constants, which consistent with the research work of Navier-Stokes equations in \cite{ZZ2014}, where they only requiring initial shear viscosity to be large. And this result improves the study of \cite{HLZ} to some extent, for they requiring initial shear and bulk viscosities all to be large.
\end{rem}

\begin{rem}
As described in \cite{HLZ}, Xin and Zhu obtained the global regular solutions of 3D compressible Navier-Stokes equations if $\alpha>2\gamma-1$. In this paper, we assume that $\alpha>\max \left\{\frac{\gamma+1}{2},\,\gamma-1\right\}$. It is obvious that $2\gamma-1>\frac{\gamma+1}{2}$ and $2\gamma-1>\gamma-1$ due to $\gamma>1$. The problem of \cite{HLZ} allows vacuum. However, the present of vacuum will bring  essential difficulties to the proof of global existence of large solutions of 3D compressible micropolar system which will be addressed in our future work.
\end{rem}

We now provide our analysis and commentary on the key aspects of this paper. Rewriting $(\ref{1.1})_1$ and $(\ref{1.1})_2$ as the following:
\begin{align}
\mu_1\Delta u+(\mu_1+\lambda_1)\nabla\dg u=&\rho^{1-\alpha} \dot u-\mu_1\alpha\rho^{-1}\nabla\rho\otimes d(u)-\lambda_1\alpha\rho^{-1}\dg u\nabla\rho \nonumber\\
&+A\gamma\rho^{\gamma-\alpha-1}\nabla\rho-\rho^{-\alpha}\dg (2\bx\rho^\beta d(u))\nonumber\\
&+\rho^{-\alpha}\cl(\bx\rho^\beta\cl u)\nonumber\\
&-2\bx\beta\rho^{\beta-\alpha-1}\na\rho\times w-2\bx\rho^{\beta-\alpha}\cl w,\label{1.12}\\
\mu_2\Delta w+(\mu_2+\lambda_2)\nabla\dg w=&\rho^{1-\alpha} \dot w-\mu_2\alpha\rho^{-1}\nabla\rho\cdot\na w-(\mu_2+\lambda_2)\alpha\rho^{-1}\dg w\nabla\rho \nonumber\\
&-2\bx\rho^{\beta-\alpha}\cl u+4\bx\rho^{\beta-\alpha}w,\label{1.13}
\end{align}
where $$\dot u\triangleq u_t+u\cdot\na u,\quad \dot w\triangleq w_t+u\cdot\na w$$
denote the material derivatives. Therefore, if $\alpha>1,\,\alpha>\gamma-1,\,\beta<\alpha$ and $\rho\gg 1$, each term of the right-hand of (\ref{1.12}) and (\ref{1.13}) can be viewed as a small perturbation term, which is also the main observations of \cite{HLZ} and \cite{Yu2025}.

The rest of this paper is organized as follows. In section 2, we collect some elementary facts and inequality that will be needed in later analysis. In Section 3, we will devote to the $\rm a priori $ estimates that are needed to obtain the global existence of strong solutions. Then the main result, Theorem {\ref{th1} is proved in Section 4. The Remark \ref{re5} will be proved briefly in section 5.
\section{Preliminaries}
In this section, we recall some known facts and elementary inequalities that will used in the following proof. We begin with the local existence of strong solutions to the Cauchy problem (\ref{1.1})--(\ref{1.5}), the proof can be similarly performed as those in \cite{ZZ2014}, we omit the detail here.
\begin{lem}\label{lem2.1}
{\rm (Local existence)}
Assume that $(\rho_0, u_0, w_0)$ satisfies (\ref{1.6}), then there exist a small time $T_0 >0$ and a unique strong solution $(\rho, u, w)$ to the problem (\ref{1.1})--(\ref{1.5}) in $\R^3 \times (0, T_0]$ satisfying \be\left\{
\begin{aligned}
&\rho\in C([0,T_0];D^{1,2}\cap D^{1,q}),\quad \rho_t\in C([0,T_0];L^q),\\
&\na u\in C([0,T_0];H^1)\cap L^2([0,T_0];W^{1,q}),\\
&\sqrt{\rho}u_t\in L^\infty([0,T_0];L^2),\quad u_t\in L^2([0,T_0];H^1),\\
& w\in C([0,T_0];H^2)\cap L^2([0,T_0];W^{2,q}),\\
&\sqrt{\rho}w_t\in L^\infty([0,T_0];L^2),\quad w_t\in L^2([0,T_0];H^1).\label{2.1}
\end{aligned}\right.
\ee
\end{lem}

Next, the following well-known Gagliardo-Nirenberg inequality will be used frequently later (see \cite{LSU1968}, Theorem 2.2).
\begin{lem}\label{lem2.2}
{\rm (Gagliardo-Nirenberg inequality)}
For $r\in(6/5,3)$ and
\be\label{2.2}
p \in \left\{\begin{matrix}\begin{aligned}
 &[2,\frac{3r}{3-r}],\\
& \left[2,\infty\right),\\
& \left[2,\infty\right],
\end{aligned}
\begin{aligned}
&\quad{ if}\,\,r\in(\frac{6}{5},3),\\
&\quad{ if}\,\,r=3,\\
&\quad{ if}\,\,r\in(3,\infty],
\end{aligned}\end{matrix}\right.\ee
 then there exists some generic positive constant $C$ that may depend on $p$ and $r$ such that for any  $f\in \left\{f|f\in L^2,\,\nabla f\in L^r\right\}$,
\be\|f\|_{L^p} \leq C(p,r) \|f\|_{L^2}^{\alpha} \|\nabla f\|_{L^r}^{1-\alpha},\quad\frac{\alpha}{2}+\left(\frac{1}{r}-\frac{1}{3}\right)(1-\alpha)=\frac{1}{p}.\label{2.2}
\ee
\end{lem}
\section{A priori estimates}
For any fixed time $T>0$, let $(\rho, u, w)$ be a strong solution to $(\ref{1.1})-(\ref{1.5})$ on $\mathbb{R}^3\times (0,T]$ with the initial data $(\rho_0, u_0, w_0)$ satisfying (\ref{1.6}). Define

\begin{align}
&\varepsilon_{\rho,1}(T)\triangleq \sup_{t\in[0,T]}\|\rho-\bar{\rho}\|_{L^\infty},\nonumber\\
&\varepsilon_{\rho,2}(T)\triangleq \sup_{t\in[0,T]}\|\nabla\rho\|_{L^q}^2+{\bar{\rho}}^{\gamma-\alpha}\int_0^T\|\nabla\rho\|_{L^q}^2\md t,\nonumber\\
&\varepsilon_{\rho,3}(T)\triangleq \sup_{t\in[0,T]}\|\nabla\rho\|_{L^2}^2+{\bar{\rho}}^{\gamma-\alpha}\int_0^T\|\nabla\rho\|_{L^2}^2\md t,\nonumber\\
&\varepsilon_{u,w,1}(T)\triangleq \sup_{t\in[0,T]}\frac{{\bar{\rho}}^\alpha}{2^{\alpha+2}}\left(\mu_1\|\nabla u\|_{L^2}^2+\mu_2\|\nabla w\|_{L^2}^2\right)\label{3.1}\\
&\qquad\qquad\,\,+\frac{1}{2}\int_0^T\|\sqrt{\rho}u_t\|_{L^2}^2\md t+\frac{1}{2}\int_0^T\|\sqrt{\rho}w_t\|_{L^2}^2\md t,\nonumber\\
&\varepsilon_{u,w,2}(T)\triangleq \sup_{t\in[0,T]}\left(\|\sqrt{\rho} u_t\|_{L^2}^2+\|\sqrt{\rho} w_t\|_{L^2}^2\right)\nonumber\\
&\qquad\qquad\,\,+\frac{{\bar{\rho}}^\alpha}{2^{\alpha+2}}\int_0^T\left(\mu_1\|\nabla u_t\|_{L^2}^2+\mu_2\|\nabla w_t\|_{L^2}^2\right)\md t.\nonumber
\end{align}

We have the following key proposition.
\begin{prop}\label{p}
Under the conditions of Theorem \ref{th1}, for
\begin{equation}\label{3.2}
b=\max\{3-\gamma,\,0\},
\end{equation}
there exist positive constants
$$N_3=N_3(\mu_1,\,\lambda_1,\,\mu_2,\,\lambda_2,\,\bx,\,\rho_0,\,\|\na u_0\|_{L^2},\,\|w_0\|_{H^1}),$$
$$N_i=N_i(A,\,\mu_1,\,\lambda_1,\,\mu_2,\,\lambda_2,\,\bx,\,\alpha,\,\beta,\,\gamma,\,\|G(\rho_0)\|_{L^1},\,\|\na \rho_0\|_{L^2},\,\|\na \rho_0\|_{L^q},\,\|\na u_0\|_{H^1},\,\|w_0\|_{H^2})$$
for $i=1,2,4$ and
$$\Lambda_0=\Lambda_0(A,\,\mu_1,\,\lambda_1,\,\mu_2,\,\lambda_2,\,\bx,\,\alpha,\,\beta,\,\gamma,\,\|G(\rho_0)\|_{L^1},\,\|\na \rho_0\|_{L^2},\,\|\na \rho_0\|_{L^q},\,\|u_0\|_{H^2},\,\|w_0\|_{H^2})$$
such that if $(\rho,\,u,\,w)$ is a smooth solution to the problem (\ref{1.1})-(\ref{1.5}) on $\mathbb{R}^3\times(0,T]$ satisfying
\be\label{3.3}
\begin{aligned}
&\varepsilon_{\rho,1}(T)\leq \frac{\bar{\rho}}{2},\quad
\varepsilon_{\rho,2}(T)\leq 2N_1{\bar{\rho}}^b,\quad
\varepsilon_{\rho,3}(T)\leq 2N_2{\bar{\rho}}^b,\\
&\varepsilon_{u,w,1}(T)\leq 2^{\alpha+2}N_3{\bar{\rho}}^\alpha,\quad
\varepsilon_{u,w,2}(T)\leq 3N_4{\bar{\rho}}^{2\alpha-1},
\end{aligned}
\ee
the following estimates hold:
\be\label{3.4}
\begin{aligned}
&\varepsilon_{\rho,1}(T)\leq \frac{\bar{\rho}}{4},\quad
\varepsilon_{\rho,2}(T)\leq N_1{\bar{\rho}}^b,\quad
\varepsilon_{\rho,3}(T)\leq N_2{\bar{\rho}}^b,\\
&\varepsilon_{u,w,1}(T)\leq 2^{\alpha+1}N_3{\bar{\rho}}^\alpha,\quad
\varepsilon_{u,w,2}(T)\leq 2N_4{\bar{\rho}}^{2\alpha-1},
\end{aligned}
\ee
provided that
$$\bar{\rho}\geq \Lambda_0,$$
with $\Lambda_0$ defined in (\ref{3.67}).
\end{prop}

First, we have the following basic energy estimate.
\begin{lem}\label{L1}
Under the assumption that
\begin{equation}\label{3.5}\varepsilon_{\rho,1}(T)\leq \frac{\bar{\rho}}{2},\end{equation}
then it holds that
\begin{align}
&\sup_{t\in[0,T]}\left(\bar{\rho}\|u\|_{L^2}^2+\bar{\rho}\|w\|_{L^2}^2+{\bar{\rho}}^{\gamma-2}\|\rho-\bar{\rho}\|_{L^2}^2\right)\nonumber\\
&\qquad+{\bar{\rho}}^{\alpha}\int_0^T\left(\|\nabla u\|_{L^2}^2+\|\nabla w\|_{L^2}^2\right)\md t+\bx{\bar{\rho}}^{\beta}\int_0^T\|w\|_{L^2}^2\md t\leq C\bar{\rho},\label{3.6}
\end{align}
with $$C=C(A,\,\gamma,\,\mu_1,\,\mu_2,\,\lambda_1,\,\lambda_2,\,\|u_0\|_{L^2},\,\|w_0\|_{L^2},\,\|G(\rho_0)\|_{L^1}).$$
\end{lem}
\pf First, it follows from (\ref{3.5}) that
\begin{equation}\label{3.7}
\frac{\bar{\rho}}{2}\leq \rho\leq \frac{3\bar{\rho}}{2}.
\end{equation}


Then, multiplying $(\ref{1.1})_2,$ $(\ref{1.1})_3$ by $u$ and $w$ respectively,  integrating the resultant equations, we obtain that after using $(\ref{1.1})_1$ and integration by parts that
\begin{align}
&\frac{\md}{\md t}\left(\frac{1}{2}\|\sqrt{\rho}u\|_{L^2}^2+\frac{1}{2}\|\sqrt{\rho}w\|_{L^2}^2+\|G(\rho)\|_{L^1}\right)+\frac{\mu_1}{2^\alpha}{\bar{\rho}}^\alpha\|\nabla u\|_{L^2}^2+\frac{\mu_2}{2^\alpha}{\bar{\rho}}^\alpha\|\nabla w\|_{L^2}^2\nonumber\\
&\quad+\frac{\mu_1+\lambda_1}{2^\alpha}{\bar{\rho}}^\alpha\|{\rm div} u\|_{L^2}^2+\frac{\mu_2+\lambda_2}{2^\alpha}{\bar{\rho}}^\alpha\|{\rm div} w\|_{L^2}^2+\frac{\bar{\xi}}{2^\alpha}{\bar{\rho}}^\beta\|{\rm curl}u-2w\|_{L^2}^2\nonumber\\
&\leq\frac{\md}{\md t}\left(\frac{1}{2}\|\sqrt{\rho}u\|_{L^2}^2+\frac{1}{2}\|\sqrt{\rho}w\|_{L^2}^2+\|G(\rho)\|_{L^1}\right)\nonumber\\
&\quad+\int\left(2\left(\mu_1{\rho}^\alpha+\bx\rho^\beta\right)|d (u)|^2+\mu_2{\rho}^\alpha|\nabla w|^2\right)\md x+\int\lambda_1{\rho}^\alpha(\dg u)^2\md x\nonumber\\
&\quad+\int(\mu_2+\lambda_2){\rho}^\alpha(\dg w)^2\md x+\int{\bar{\xi}}{\rho}^\beta|{\rm curl}u-2w|^2\md x=0,\label{3.9}
\end{align}
where
\begin{equation}\label{3.10}
G(\rho)=\rho\int_{\br}^{\rho}\frac{P(s)-P(\br)}{s^2}\md s.
\end{equation}

Integrating (\ref{3.9}) over $(0,T]$, (\ref{3.7}) and H$\rm{\ddot{o}}$lder inequality leads to
\begin{align}
&\sup_{t\in[0,T]}\left(\bar{\rho}\|u\|_{L^2}^2+\bar{\rho}\|w\|_{L^2}^2+{\bar{\rho}}^{\gamma-2}\|\rho-\bar{\rho}\|_{L^2}^2\right)+{\bar{\rho}}^\alpha\int_0^T\left(\|\nabla u\|_{L^2}^2+\|\nabla w\|_{L^2}^2\right)\md t \nonumber\\
&\qquad+\bx{\bar{\rho}}^\beta\int_0^T\|{\rm curl}u-2w\|_{L^2}^2\md t\nonumber\\
&\quad \leq C\left(\bar{\rho}\|u_0\|_{L^2}^2+\bar{\rho}\|w_0\|_{L^2}^2+\|G(\rho_0)\|_{L^1}\right),\label{3.11}
\end{align}
where we have used the fact
$$G(\rho)\sim {\bar{\rho}}^{\gamma-2}(\rho-\bar{\rho})^2$$
from (\ref{3.7}) and (\ref{3.10}). (\ref{3.6}) follows (\ref{3.11}) directly.
\hfill $\square$

Next, following $(\ref{1.12})$ and $(\ref{1.13})$ that
\begin{equation}\label{3.12}
\left\{\begin{aligned}
&(2\mu_1+\lambda_1)\nabla\dg u-\mu_1\cl\cl u=\nabla \mathcal{P}+H_1,\\
&(2\mu_2+\lambda_2)\nabla\dg w-\mu_2\cl\cl w=H_2,
\end{aligned}\right.
\end{equation}
where
\be\label{3.13}
\nabla \mathcal{P}=\left\{
\begin{aligned}
&\frac{A\gamma}{\gamma-\alpha}\nabla\rho^{\gamma-\alpha},\quad&\gamma\neq\alpha,\\
&A\gamma\nabla{\rm ln}\rho,\quad&\gamma=\alpha,
\end{aligned}
\right.
\ee
and

\begin{align}
H_1\triangleq&\rho^{1-\alpha} \dot u-\mu_1\alpha\rho^{-1}\nabla\rho\otimes d(u)-\lambda_1\alpha\rho^{-1}\dg u\nabla\rho \nonumber\\
&+A\gamma\rho^{\gamma-\alpha-1}\nabla\rho-\rho^{-\alpha}\dg (2\bx\rho^\beta d(u))+\rho^{-\alpha}\cl(\bx\rho^\beta\cl u)\nonumber\\
&-2\bx\beta\rho^{\beta-\alpha-1}\na\rho\times w-2\bx\rho^{\beta-\alpha}\cl w,\label{3.14}\\
H_2\triangleq&\rho^{1-\alpha}\dot{w}-\mu_2\alpha\rho^{-1}\nabla\rho\cdot\nabla w-(\mu_2+\lambda_2)\alpha\rho^{-1}\nabla\rho\dg w\nonumber\\
&-2\bar{\xi}\rho^{\beta-\alpha}\cl u+4\bar{\xi}\rho^{\beta-\alpha} w.\label{3.15}
\end{align}

To proceed, similar to the compressible Navier-Stokes equations (see, for example,\cite{HLX2012}, we introduce the effective viscous flux for the micropolar fuids related to the velocity, the micro-rotational velocity, and the pressure:
\be\label{3.16}\left\{
\begin{aligned}
&F_1=(2\mu_1+\lambda_1)\dg u-\mathcal{P}(\rho)+\mathcal{P}({\bar{\rho}}),\quad &V_1=\mu_1\cl u,\\
&F_2=(2\mu_2+\lambda_2)\dg w,\quad &V_2=\mu_2\cl w,
\end{aligned}\right.\end{equation}
where $F_1$ and $F_2$ are the so-called "effective viscous flux", and $V_1$ and $V_2$ are the vorticity. Then (\ref{3.12}) and (\ref{3.16}) imply that
\begin{align}
\Delta F_1=\dg H_1,\quad \Delta V_1=\cl H_1,\nonumber\\
\Delta F_2=\dg H_2,\quad \Delta V_2=\cl H_2.\nonumber
\end{align}
According to the standard $L^2$-estimate of the elliptic system, after using (\ref{3.7}) and Sobolev inequality, we have
\begin{align}
&\|\nabla^2 u\|_{L^2}+\|\nabla^2 w\|_{L^2}\nonumber\\
&\leq C\left(\|\nabla F_1\|_{L^2}+\|\nabla V_1\|_{L^2}+{\bar{\rho}}^{-\alpha}\|\nabla P(\rho)\|_{L^2}+\|\nabla F_2\|_{L^2}+\|\nabla V_2\|_{L^2}\right)\nonumber\\
&\leq C\left(\|H_1\|_{L^2}+\|H_2\|_{L^2}+{\bar{\rho}}^{-\alpha}\|\nabla P(\rho)\|_{L^2}\right)\nonumber\\
&\leq C{\bar{\rho}}^{-\alpha+\frac{1}{2}}\|\sqrt{\rho}u_t\|_{L^2}+C{\bar{\rho}}^{-\alpha+1}\|\nabla u\|_{L^2}^{\frac{3}{2}}\|\nabla u\|_{L^6}^{\frac{1}{2}}+C{\bar{\rho}}^{-1}\|\nabla \rho\|_{L^q}\|\nabla u\|_{L^2}^{\frac{q-3}{q}}\|\nabla u\|_{L^6}^{\frac{3}{q}}\nonumber\\
&\quad+C{\bar{\rho}}^{-\alpha+\gamma-1}\|\nabla \rho\|_{L^2}+C{\bar{\rho}}^{-\alpha+\frac{1}{2}}\|\sqrt{\rho}w_t\|_{L^2}+C{\bar{\rho}}^{-\alpha+1}\|\nabla u\|_{L^2}\|\nabla w\|_{L^2}^{\frac{1}{2}}\|\nabla w\|_{L^6}^{\frac{1}{2}}\nonumber\\
&\quad+C{\bar{\rho}}^{-1}\|\nabla \rho\|_{L^q}\|\nabla w\|_{L^2}^{\frac{q-3}{q}}\|\nabla w\|_{L^6}^{\frac{3}{q}}+C\br^{\beta-\alpha}\|\nabla^2u\|_{L^2}\nonumber\\
&\quad+C\bx\br^{\beta-\alpha}(\|\nabla u\|_{L^2}+\|\nabla w\|_{L^2}+\|w\|_{L^2})+C\br^{\beta-\alpha-1}\|\na\rho\|_{L^3}\|w\|_{L^6}\nonumber\\
&\leq \frac{1}{4}\|\nabla^2 u\|_{L^2}+\frac{1}{2}\|\nabla^2 w\|_{L^2}+C_1\br^{\beta-\alpha}\|\na^2u\|_{L^2}+C\br^{-\alpha+\frac{1}{2}}\left(\|\sqrt{\rho}u_t\|_{L^2}+\|\sqrt{\rho}w_t\|_{L^2}\right)\nonumber\\
&\quad+C{\bar{\rho}}^{-2\alpha+2}\left(\|\nabla u\|_{L^2}^3+\|\nabla w\|_{L^2}^3\right)+C\br^{-\alpha+\gamma-1}\|\na\rho\|_{L^2}\nonumber\\
&\quad+C\br^{-\frac{q}{q-3}}\|\na\rho\|_{L^q}^{\frac{q}{q-3}}\left(\|\na u\|_{L^2}+\|\na w\|_{L^2}\right)+C\bx\br^{\beta-\alpha}(\|\nabla u\|_{L^2}+\|\nabla w\|_{L^2}+\| w\|_{L^2}).\label{3.17}
\end{align}
where we have used
$$\|\nabla F_1\|_{L^2}+\|\nabla V_1\|_{L^2}\leq C\|H\|_{L^2},\quad \|\nabla F_2\|_{L^2}+\|\nabla V_2\|_{L^2}\leq C\|H_2\|_{L^2}.$$

Similarly, the standard $L^p-$estimate of the elliptic system and Gagliardo-Nirenberg inequality Lemma \ref{lem2.2} gives that
\begin{align}
&\|\nabla^2 u\|_{L^q}+\|\nabla^2 w\|_{L^q}\nonumber\\
&\leq C\left(\|\nabla F_1\|_{L^q}+\|\nabla V_1\|_{L^q}+{\bar{\rho}}^{-\alpha}\|\nabla P(\rho)\|_{L^q}+\|\nabla F_2\|_{L^q}+\|\nabla V_2\|_{L^q}\right)\nonumber\\
&\leq C{\bar{\rho}}^{-\alpha}\|\rho u_t\|_{L^q}+C{\bar{\rho}}^{-\alpha+1}\| u\|_{L^q}\|\nabla u\|_{L^\infty}+C{\bar{\rho}}^{-1}\|\nabla \rho\|_{L^q}\|\nabla u\|_{L^\infty}\nonumber\\
&\quad+C{\bar{\rho}}^{-\alpha+\gamma-1}\|\nabla \rho\|_{L^q}+C{\bar{\rho}}^{-\alpha}\|\rho w_t\|_{L^q}+C{\bar{\rho}}^{-\alpha+1}\| u\|_{L^q}\|\nabla w\|_{L^\infty}\nonumber\\
&\quad+C{\bar{\rho}}^{-1}\|\nabla \rho\|_{L^q}\|\nabla w\|_{L^\infty}+C\br^{\beta-\alpha}\|\nabla^2u\|_{L^q}+C\bx\br^{\beta-\alpha}(\|\nabla u\|_{L^q}+\|\nabla w\|_{L^q}+\|w\|_{L^q})\nonumber\\
&\leq C{\bar{\rho}}^{-\alpha}\|\rho u_t\|_{L^q}+C{\bar{\rho}}^{-\alpha+1}\| u\|_{L^q}\|\nabla u\|_{L^\infty}+C{\bar{\rho}}^{-1}\|\nabla \rho\|_{L^q}\|\nabla u\|_{L^\infty}\nonumber\\
&\quad+C{\bar{\rho}}^{-\alpha+\gamma-1}\|\nabla \rho\|_{L^q}+C{\bar{\rho}}^{-\alpha}\|\rho w_t\|_{L^q}+C{\bar{\rho}}^{-\alpha+1}\| u\|_{L^q}\|\nabla w\|_{L^\infty}\nonumber\\
&\quad+C{\bar{\rho}}^{-1}\|\nabla \rho\|_{L^q}\|\nabla w\|_{L^\infty}+C\br^{\beta-\alpha}\|\nabla^2u\|_{L^q}\nonumber\\
&\quad+C\br^{\beta-\alpha}(\|\nabla u\|_{L^2}^{\theta}\|\nabla^2 u\|_{L^q}^{1-\theta}+\|\nabla w\|_{L^2}^{\theta}\|\nabla^2 w\|_{L^q}^{1-\theta}+\|w\|_{L^q})\nonumber\\
&\leq C_2\br^{\beta-\alpha}\|\na^2u\|_{L^q}+C\br^{-\alpha}\left(\|\rho u_t\|_{L^q}+\|\rho w_t\|_{L^q}\right)\nonumber\\
&\quad+C{\bar{\rho}}^{-\alpha+1}\| u\|_{L^q}\|\nabla u\|_{L^\infty}+C{\bar{\rho}}^{-1}\|\nabla \rho\|_{L^q}\|\nabla u\|_{L^\infty}\nonumber\\
&\quad+C{\bar{\rho}}^{-\alpha+\gamma-1}\|\nabla \rho\|_{L^q}+C{\bar{\rho}}^{-\alpha+1}\| u\|_{L^q}\|\nabla w\|_{L^\infty}
+C{\bar{\rho}}^{-1}\|\nabla \rho\|_{L^q}\|\nabla w\|_{L^\infty}\nonumber\\
&\quad+C\bx^{\frac{1}{\theta}}\br^{(\beta-\alpha)\frac{1}{\theta}}\left(\|\nabla u\|_{L^2}+\|\nabla w\|_{L^2}\right)+C\bx\br^{\beta-\alpha}\|w\|_{L^q}.\label{3.18}
\end{align}
where
\begin{equation}\label{3.19}
\theta=\frac{2q}{5q-6}.
\end{equation}


As a consequence, we have the following estimates, which will be used in the following proofs frequently.
\begin{lem}\label{L2}
Let
\begin{equation}\label{3.20}
\alpha>1,\,\, 0\leq \beta< \alpha,\,\,\frac{b}{2}\leq \min\left\{1, \alpha-\gamma+1,\frac{\alpha-\gamma+2}{2}\right\},
\end{equation}
Under the assumptions that (\ref{3.3}), then there exists a positive constant $M_1$ such that
\begin{align}
&\sup_{t\in[0,T]}\left(\|\na u\|_{L^6}+\|\na w\|_{L^6}\right)\leq C,\label{3.21}\\
&\int_0^T\left(\|\na u\|_{L^6}^2+\|\na w\|_{L^6}^2\right)\md t\leq C,\label{3.22}
\end{align}
provided $\bar{\rho}\geq M_1(\mu_1,\,\mu_2,\,\alpha,\,\beta,\,\bx)$, where $C=C(A,\,\mu_1,\,\mu_2,\,\lambda_1,\,\lambda_2,\,\bx,$ $\alpha,\,\beta,\,\gamma,$ $\|u_0\|_{H^2},$ $\|w_0\|_{H^2},$ $\|G(\rho_0\|_{L^1},$ $\|\na\rho_0\|_{L^2},$ $\|\na\rho_0\|_{L^q}).$
\end{lem}
\pf
Combing (\ref{3.3}), (\ref{3.6}) with (\ref{3.17}), and taking $\br\geq M_1\triangleq \left(2C_1\right)^{\frac{1}{\alpha-\beta}}$ gives that
\begin{align}
&\|\nabla u\|_{L^6}+\|\nabla w\|_{L^6}\nonumber\\
&\leq C\br^{-\alpha+\frac{1}{2}}\left(\|\sqrt{\rho}u_t\|_{L^2}+\|\sqrt{\rho}w_t\|_{L^2}\right)\nonumber\\
&\quad+C{\bar{\rho}}^{-2\alpha+2}\left(\|\nabla u\|_{L^2}^3+\|\nabla w\|_{L^2}^3\right)+C\br^{-\alpha+\gamma-1}\|\na\rho\|_{L^2}\nonumber\\
&\quad+C\br^{-\frac{q}{q-3}}\|\na\rho\|_{L^q}^{\frac{q}{q-3}}\left(\|\na u\|_{L^2}+\|\na w\|_{L^2}\right)+C\br^{\beta-\alpha}(\|\nabla u\|_{L^2}+\|\nabla w\|_{L^2}+\| w\|_{L^2})\nonumber\\
&\leq C+C{\bar{\rho}}^{-2\alpha+2}+C{\bar{\rho}}^{-\alpha+\gamma-1+\frac{b}{2}}+C\br^{-\frac{q}{q-3}\left(1-\frac{b}{2}\right)}+C\br^{\beta-\alpha}\leq C\label{3.23}
\end{align}
since $\alpha>1,\,b\leq 2, \,\beta<\alpha$ and
$$-\alpha+\gamma-1+\frac{b}{2}\leq 0$$
due to (\ref{3.20}).

Next, also taking $\br\geq M_1$, it follows (\ref{3.3}), (\ref{3.6}) and (\ref{3.17}) that
\begin{align}
&\int_0^T\left(\|\nabla u\|_{L^6}^2+\|\nabla w\|_{L^6}^2\right)\md t\nonumber\\
&\leq C\br^{-2\alpha+1}\int_0^T\left(\|\sqrt{\rho}u_t\|_{L^2}^2+\|\sqrt{\rho}w_t\|_{L^2}^2\right)\md t\nonumber\\
&\quad+C{\bar{\rho}}^{-4\alpha+4}\int_0^T\left(\|\nabla u\|_{L^2}^2+\|\nabla w\|_{L^2}^2\right)\md t+C\br^{2(-\alpha+\gamma-1)}\int_0^T\|\na\rho\|_{L^2}^2\md t\nonumber\\
&\quad+C\br^{-\frac{2q}{q-3}}\int_0^T\|\na\rho\|_{L^q}^{\frac{2q}{q-3}}\left(\|\na u\|_{L^2}^2+\|\na w\|_{L^2}^2\right)\md t\nonumber\\
&\quad+C\br^{2(\beta-\alpha)}\int_0^T\left(\|\nabla u\|_{L^2}^2+\|\nabla w\|_{L^2}^2+\| w\|_{L^2}^2\right)\md t\nonumber\\
&\leq C\left(\br^{1-\alpha}+{\bar{\rho}}^{5-5\alpha}+C{\bar{\rho}}^{-\alpha+\gamma-2+b}+\br^{\frac{q}{q-3}(b-2)+1-\alpha}+\br^{2(\beta-\alpha)+1-\alpha}+\br^{\beta-\alpha+1-\alpha}\right)\nonumber\\
&\leq C\br^{1-\alpha}+C{\bar{\rho}}^{-\alpha+\gamma-2+b}\leq C\label{3.24}
\end{align}
since $\alpha>1,\,b\leq 2,\,\beta<\alpha$ and
$$
-\alpha+\gamma-2+b\leq 0$$
due to (\ref{3.20}).
\hfill $\square$

Now, let us concentrate on the estimate of velocities.
\begin{lem}\label{L3}
Let
\begin{equation}\label{3.25}
\alpha>\max\left\{1,\frac{\gamma+1}{2}\right\},\,\, 0\leq \beta<\alpha,\,\,b<2.
\end{equation}
Under the assumptions that (\ref{3.3}), then there exists a positive constant $M_2$ such that
\begin{equation}\label{3.26}
\varepsilon_{u,w,1}\leq 2^{\alpha+1}N_3\br^\alpha
\end{equation}
provided $\bar{\rho}\geq M_2(A,\alpha,\,\beta,\,\gamma,$ $\mu_1,\,\mu_2,\,\lambda_1,\,\lambda_2$, $\bx$, $G(\rho_0)$, $\|\nabla\rho_0\|_{L^2}$, $\|\nabla\rho_0\|_{L^q}$, $\|u_0\|_{H^2}$, $\|w_0\|_{H^2})$. Here $N_3$ is defined in (\ref{3.32}).
\end{lem}

\pf Multiplying $(\ref{1.1})_2$ and $(\ref{1.1})_3$ by $u_t$ and $w_t$, respectively, and integrating by parts yields that
\begin{align}
&\frac{\md}{\md t}\int\rho^\alpha\left(\mu_1|d(u)|^2+\frac{\lambda_1}{2}(\dg u)^2+\frac{\mu_2}{2}|\na w|^2+\frac{\mu_2+\lambda_2}{2}(\dg w)^2\right)\md x\nonumber\\
&\quad+\frac{\md}{\md t}\int\bx\rho^\beta|d(u)|^2\md x+\frac{1}{2}\frac{\md}{\md t}\int\bx\rho^\beta\left|\cl u-2w\right|^2\md x\nonumber\\
&\quad+\int\rho\left(|u_t|^2+|w_t|^2\right)\md x\nonumber\\
&=\frac{\md}{\md t}\int \left(P(\rho)-P(\bar{\rho})\right)\dg u\md x-\int{\rho}u\cdot\na u\cdot u_t\md x\nonumber\\
&\quad-\int{\rho}u\cdot\na w\cdot w_t\md x+\int(\rho^\alpha)_t\left(\mu_1|d(u)|^2+\frac{\lambda_1}{2}(\dg u)^2\right)\md x\nonumber\\
&\quad+\frac{1}{2}\int(\rho^\alpha)_t\left(\mu_2|\na w|^2+(\mu_2+\lambda_2)(\dg w)^2\right)\md x+\int\bx(\rho^\beta)_t|d(u)|^2\md x\nonumber\\
&\quad+\frac{1}{2}\int\bx\left(\rho^\beta\right)_t\left|\cl u-2w\right|^2\md x-\int P_t\dg u\md x\nonumber\\
&\triangleq\frac{\md}{\md t}\int \left(P(\rho)-P(\bar{\rho})\right)\dg u\md x+\sum_{i=1}^{7}I_i.\label{3.27}
\end{align}

It follows from (\ref{3.3}), (\ref{3.6}), (\ref{3.7}), (\ref{3.21}), H${\rm \ddot{o}}$lder inequality and Gagliardo-Nirenberg inequality that
\begin{align}
I_1+I_2&\leq \frac{1}{4}\|\sqrt{\rho}u_t\|_{L^2}^2+\frac{1}{2}\|\sqrt{\rho}w_t\|_{L^2}^2+C\br\|u\|_{L^6}^2\|\na u\|_{L^3}^2+C\br\|u\|_{L^6}^2\|\na w\|_{L^3}^2\nonumber\\
&\leq \frac{1}{4}\|\sqrt{\rho}u_t\|_{L^2}^2+\frac{1}{2}\|\sqrt{\rho}w_t\|_{L^2}^2+C\br\|\na u\|_{L^2}^3\|\na u\|_{L^6}\nonumber\\
&\quad+C\br\|\na u\|_{L^2}^2\|\na w\|_{L^2}\|\na w\|_{L^6}\nonumber\\
&\leq \frac{1}{2}\|\sqrt{\rho}u_t\|_{L^2}^2+\frac{1}{2}\|\sqrt{\rho}w_t\|_{L^2}^2+C\br\|\na u\|_{L^2}^2,\label{3.28}
\end{align}
\begin{align}
\sum_{i=3}^{6}I_i&\leq C\br^{\alpha-1}\int|\rho\dg u+\na\rho\cdot u|\left(|\na u|^2+|\na w|^2\right)\md x\nonumber\\
&\quad+C\bx\br^{\beta-1}\int|\rho\dg u+\na\rho\cdot u||w|^2\md x\nonumber\\
&\leq C\br^{\alpha}\|\na u\|_{L^3}^3+C\br^{\alpha}\|\na u\|_{L^3}\|\na w\|_{L^3}^2\nonumber\\
&\quad+C\br^{\alpha-1}\|\na\rho\|_{L^q}\|u\|_{L^6}\|\na u\|_{L^{\frac{12q}{5q-6}}}^2+C\br^{\alpha-1}\|\na\rho\|_{L^q}\|u\|_{L^6}\|\na w\|_{L^{\frac{12q}{5q-6}}}^2\nonumber\\
&\quad+C\bx\br^{\beta}\|\na u\|_{L^2}\|w\|_{L^3}\|w\|_{L^6}+C\bx\br^{\beta-1}\|\na\rho\|_{L^3}\|u\|_{L^6}\| w\|_{L^6}^2\nonumber\\
&\leq C\br^{\alpha}\|\na u\|_{L^2}^{\frac{3}{2}}\|\na u\|_{L^6}^{\frac{3}{2}}+C\br^{\alpha}\|\na w\|_{L^2}^{\frac{3}{2}}\|\na w\|_{L^6}^{\frac{3}{2}}\nonumber\\
&\quad+C\br^{\alpha-1+\frac{b}{2}}\|\na u\|_{L^2}^{\frac{5q-6}{2q}}\|\na u\|_{L^6}^{\frac{q+6}{2q}}\nonumber\\
&\quad+C\br^{\alpha-1+\frac{b}{2}}\|\na u\|_{L^2}\|\na w\|_{L^2}^{\frac{3q-6}{2q}}\|\na w\|_{L^6}^{\frac{q+6}{2q}}\nonumber\\
&\quad+C\bx\br^{\beta}\|\na u\|_{L^2}\|w\|_{L^2}^{\frac{1}{2}}\|\na w\|_{L^2}^{\frac{3}{2}}+C\bx\br^{\beta-1+\frac{b}{2}}\|\nabla u\|_{L^2}\|\nabla w\|_{L^2}^2,\label{3.29}\end{align}
and
\begin{align}
I_7&\leq C\br^{\gamma-1}\int|\rho\dg u+\na\rho\cdot u||\dg u|\md x\nonumber\\
&\leq C\br^{\gamma}\|\na u\|_{L^2}^2+C\br^{\gamma-1}\|\na\rho\|_{L^q}\|u\|_{L^{\frac{2q}{q-2}}}\|\na u\|_{L^2}\nonumber\\
&\leq C\br^{\gamma}\|\na u\|_{L^2}^2+C\br^{\gamma-1}\|\na\rho\|_{L^q}\|u\|_{L^2}^{\frac{q-3}{q}}\|\na u\|_{L^2}^{\frac{q+3}{q}}\nonumber\\
&\leq C\br^{\gamma}\|\na u\|_{L^2}^2+C\br^{\gamma-1}\|\na\rho\|_{L^q}\|\na u\|_{L^2}\nonumber\\
&\leq C\br^{\gamma}\|\na u\|_{L^2}^2+C\br^{\gamma-2}\|\na\rho\|_{L^q}^2.\label{3.30}
\end{align}

Substituting (\ref{3.28})$-$(\ref{3.30}) into (\ref{3.27}), integrating with respect to $t$, and using (\ref{3.6}), (\ref{3.7}), (\ref{3.22}) and (\ref{3.25}) gives
\begin{align}
&\frac{\br^\alpha}{2^{\alpha+1}}\sup_{t\in[0,T]}\int\left(\mu_1|\na u|^2+\mu_2|\na w|^2+(\mu_1+\lambda_1)(\dg u)^2+(\mu_2+\lambda_2)(\dg w)^2\right)\nonumber\\
&\qquad+\frac{1}{2}\int_0^T\|\sqrt{\rho}u_t\|_{L^2}^2\md t+\frac{1}{2}\int_0^T\|\sqrt{\rho}w_t\|_{L^2}^2\md t\nonumber\\
&\quad\leq \sup_{t\in[0,T]}\int\rho^\alpha\left(\mu_1|d(u)|^2+\frac{\mu_2}{2}|\na w|^2+\frac{\lambda_1}{2}(\dg u)^2+\frac{\mu_2+\lambda_2}{2}(\dg w)^2\right)\nonumber\\
&\qquad+\sup_{t\in[0,T]}\int\bx\rho^\beta|d(u)|^2+\frac{1}{2}\sup_{t\in[0,T]}\int\bx\rho^\beta\left|\cl u-2w\right|^2\nonumber\\
&\qquad+\frac{1}{2}\int_0^T\|\sqrt{\rho}u_t\|_{L^2}^2\md t+\frac{1}{2}\int_0^T\|\sqrt{\rho}w_t\|_{L^2}^2\md t\nonumber\\
&\quad\leq N_32^\alpha\br^\alpha+\int(P(\rho)-P(\br))\dg u\md x\nonumber\\
&\qquad+C\br^\alpha\int_0^T\left(\|\na u\|_{L^2}^{\frac{3}{2}}\|\na u\|_{L^6}^{\frac{3}{2}}+\|\na w\|_{L^2}^{\frac{3}{2}}\|\na w\|_{L^6}^{\frac{3}{2}}\right)\md t\nonumber\\
&\qquad+C\br^{\alpha}\int_0^T\|\na u\|_{L^2}^{\frac{5q-6}{2q}}\|\na u\|_{L^6}^{\frac{q+6}{2q}}\md t+C\br^{\alpha}\int_0^T\|\na u\|_{L^2}\|\na w\|_{L^2}^{\frac{3q-6}{2q}}\|\na u\|_{L^6}^{\frac{q+6}{2q}}\md t\nonumber\\
&\qquad+C\br^{\beta}\int_0^T\left(\|\na u\|_{L^2}^2+\|\na w\|_{L^2}^2\right)\md t\nonumber\\
&\qquad+C\br^{\gamma}\int_0^T\|\na u\|_{L^2}^2\md t+C\br^{\gamma-2}\int_0^T\|\na \rho\|_{L^q}^2\md t\nonumber\\
&\quad\leq N_3 2^\alpha\br^\alpha+\|P(\rho)-P(\br)\|_{L^2}\|\dg u\|_{L^2}\nonumber\\
&\qquad+C\left(\br^{\frac{\alpha+3}{4}}+\br^{\beta+1-\alpha}+\br^{\gamma+1-\alpha}+\br^{\alpha+b-2}\right)\nonumber\\
&\quad\leq N_3 2^\alpha\br^\alpha+C\br^{\gamma-1}\|\rho-\br\|_{L^2}+C\left(\br^{\frac{\alpha+3}{4}}+\br^{\gamma+1-\alpha}+\br^{\alpha+b-2}\right)\nonumber\\
&\quad \leq N_3 2^\alpha\br^\alpha+C_3\left(\br^{\frac{\gamma+1}{2}}+\br^{\frac{\alpha+3}{4}}+\br^{\gamma+1-\alpha}+\br^{\alpha+b-2}\right)\leq N_3 2^{\alpha+1}\br^\alpha,\label{3.31}
\end{align}
after defining
\begin{align}
N_3\triangleq &2\mu_1\|d( u_0)\|_{L^2}^2+\mu_2\|\na w_0\|_{L^2}^2+\lambda_1\|\dg u_0\|_{L^2}^2+(\mu_2+\lambda_2)\|\dg w_0\|_{L^2}\nonumber\\
&+2\bx\rho_0^{\beta-\alpha}\|d( u_0)\|_{L^2}^2+\bx\rho_0^{\beta-\alpha}\|\cl u_0-2w_0\|_{L^2}^2,\label{3.32}\end{align}
and choosing $M_2$ such that
\begin{align*}
\br\geq M_2\triangleq \max&\left\{M_1,\,\,\left(\frac{4C_3}{2^\alpha N_3}\right)^{\frac{2}{2\alpha-\gamma-1}},\,\,\left(\frac{4C_3}{2^\alpha N_3}\right)^{\frac{4}{3(\alpha-1)}}\left(\frac{4C_3}{2^\alpha N_3}\right)^{\frac{1}{2\alpha-\gamma-1}},\,\left(\frac{4C_3}{2^\alpha N_3}\right)^{\frac{1}{2-b}}\right\},
\end{align*}
as long as $$\alpha>\max\left\{1,\frac{\gamma+1}{2}\right\},$$
$\beta<\alpha$ and $b<2$.  The proof of Lemma \ref{L3} is complete.
\hfill $\square$
\begin{lem}\label{L4}
Assume $$\alpha>\max\left\{1,\,\frac{\gamma+1}{2}\right\},\quad 0\leq \beta<\alpha,\quad b<2,$$
and $$\gamma-1\leq \alpha.$$
There exists a positive constant $M_3$, such that
$$\varepsilon_{u,w,2}(T)\leq 2N_4\br^{2\alpha-1},$$
provided that $\br>M_3(A,\alpha,\,\beta,\,\gamma,$ $\mu_1,\,\mu_2,\,\lambda_1,\,\lambda_2$, $\bx$, $G(\rho_0)$, $\|\nabla\rho_0\|_{L^2}$, $\|\nabla\rho_0\|_{L^q}$, $\|u_0\|_{H^2}$, $\|w_0\|_{H^2})$ and $N_4$ is defined in (\ref{3.44}).
\end{lem}
\pf Differentiating $(\ref{1.1})_2$ and $(\ref{1.1})_3$ with respect to $t$, multiplying the result equations by $u_t$ and $w_t$ respectively , and integrating by parts, we have\label{3.26}
\begin{align}
&\frac{1}{2}\frac{\md}{\md t}\int\left(\rho|u_t|^2+\rho|w_t|^2\right)\md x+\int\rho^\alpha\left(2\mu_1|d(u_t)|^2+\mu_2|\na w_t|^2\right)\md x\nonumber\\
&\qquad+\int\rho^\alpha\left(\lambda_1(\dg u_t)^2+(\mu_2+\lambda_2)(\dg w_t)^2\right)\md x\\
&\qquad+\int2\bx \rho^\beta |d(u_t)|^2\md x
+\bx\int\rho^\beta|\cl u_t-2w_t|^2\md x\nonumber\\
&\quad=\int\dg(\rho u)|u_t|^2\md x-\int\rho u_t\cdot\na u\cdot u_t\md x+\int\dg(\rho u)|w_t|^2\md x\nonumber\\
&\qquad-\int\rho u_t\cdot\na w\cdot w_t\md x+\int\dg(\rho u)u\cdot\na u\cdot u_t\md x+\int\dg(\rho u)u\cdot\na w\cdot w_t\md x\nonumber\\
&\qquad-\int2 \mu_1\left(\rho^\alpha\right)_td(u):d( u_t)\md x-\int \mu_2(\rho^\alpha)_t\na w:\na w_t\md x\nonumber\\
&\qquad-\int \lambda_1\left(\rho^\alpha\right)_t\dg u\dg u_t\md x-\int (\mu_2+\lambda_2)(\rho^\alpha)_t\dg w\dg w_t\md x\nonumber\\
&\qquad-\int 2\bx\left(\rho^\beta\right)_td(u):d( u_t)\md x-\frac{1}{2}\int\bx\left(\rho^\beta\right)_t\left(\left|\cl u-2w\right|^2\right)_t\md x\nonumber\\
&\qquad+\int P_t\dg u_t\md x\triangleq\sum_{i=1}^{13}J_i.\label{3.33}
\end{align}
It follows from H$\rm\ddot{o}$lder inequality, Gagliardo-Nirenberg inequality, (\ref{3.3}) and (\ref{3.6}) that
\begin{align}
&J_1+J_2+J_3+J_4\nonumber\\
&\leq C\int\left(\rho|\na u|+\rho|\na w|+|\na\rho\cdot u|\right)\left(|u_t|^2+|w_t|^2\right)\md x\nonumber\\
&\leq C\left(\br\|\na u\|_{L^2}+\br\|\na w\|_{L^2}+\|\na\rho\|_{L^q}\|u\|_{\frac{2q}{q-2}}\right)\left(\|u_t\|_{L^4}^2+\|w_t\|_{L^4}^2\right)\nonumber\\
&\leq C\left(\br^{\frac{3}{4}}\|\na u\|_{L^2}+\br^{\frac{3}{4}}\|\na w\|_{L^2}+\br^{-\frac{1}{4}}\|\na\rho\|_{L^q}\|u\|_{L^2}^{\frac{q-3}{q}}\|\na u\|_{L^2}^{\frac{3}{q}}\right)\nonumber\\
&\quad\times\left(\|\sqrt{\rho}u_t\|_{L^2}^{\frac{1}{2}}\|\na u_t\|_{L^2}^{\frac{3}{2}}+\|\sqrt{\rho}w_t\|_{L^2}^{\frac{1}{2}}\|\na w_t\|_{L^2}^{\frac{3}{2}}\right)\nonumber\\
&\leq \frac{\mu_1}{2^{\alpha+3}}\br^\alpha\|\na u_t\|_{L^2}^2+\frac{\mu_2}{2^{\alpha+3}}\br^\alpha\|\na w_t\|_{L^2}^2+C\left(\|\sqrt{\rho}u_t\|_{L^2}^2+\|\sqrt{\rho}w_t\|_{L^2}^2\right)\nonumber\\
&\quad\times\left(\br^{3-3\alpha}\|\na u\|_{L^2}^4+\br^{3-3\alpha}\|\na w\|_{L^2}^4+\br^{-3\alpha-1}\|\na\rho\|_{L^q}^4\|\na u\|_{L^2}^{\frac{12}{q}}\right)\nonumber\\
&\leq \frac{\mu_1}{2^{\alpha+3}}\br^\alpha\|\na u_t\|_{L^2}^2+\frac{\mu_2}{2^{\alpha+3}}\br^\alpha\|\na w_t\|_{L^2}^2+C\left(\|\sqrt{\rho}u_t\|_{L^2}^2+\|\sqrt{\rho}w_t\|_{L^2}^2\right)\nonumber\\
&\quad\times\left(\br^{3-3\alpha}\|\na u\|_{L^2}^2+\br^{3-3\alpha}\|\na w\|_{L^2}^2+\br^{-3\alpha-1+2b}\|\na u\|_{L^2}^{2}\right)\nonumber\\
&\leq \frac{\mu_1}{2^{\alpha+3}}\br^\alpha\|\na u_t\|_{L^2}^2+\frac{\mu_2}{2^{\alpha+3}}\br^\alpha\|\na w_t\|_{L^2}^2+C\br^{2-\alpha}\left(\|\na u\|_{L^2}^2+\|\na w\|_{L^2}^2\right),
\label{3.34}
\end{align}
where we have used $b<2.$ Similarly, we can deduce from (\ref{3.3}), (\ref{3.6}) and (\ref{3.21}) that
\begin{align}
J_5+J_6&\leq C\int|\rho\dg u+\na\rho\cdot u||u|\left(|\na u||u_t|+|\na w||w_t|\right)\md x\nonumber\\
&\leq C\br \left(\|u\|_{L^6}\|\na u\|_{L^3}^2\|u_t\|_{L^6}+\|u\|_{L^6}\|\na u\|_{L^3}\|\na w\|_{L^3}\|w_t\|_{L^6}\right)\nonumber\\
&\quad+\left(\|\na\rho\|_{L^q}\|u\|_{L^6}^2\|\na u\|_{\frac{2q}{q-2}}\|u_t\|_{L^6}+\|\na\rho\|_{L^q}\|u\|_{L^6}^2\|\na w\|_{\frac{2q}{q-2}}\|w_t\|_{L^6}\right)\nonumber\\
&\leq \frac{\mu_1}{2^{\alpha+3}}\br^\alpha\|\na u_t\|_{L^2}^2+\frac{\mu_2}{2^{\alpha+3}}\br^\alpha\|\na w_t\|_{L^2}^2\nonumber\\
&\quad+C\br^{2-\alpha}\|\na u\|_{L^2}^4\|\na u\|_{L^6}^2+\br^{2-\alpha}\|\na u\|_{L^2}^2\|\na w\|_{L^2}^2\|\na w\|_{L^6}^2\nonumber\\
&\quad+C\br^{-\alpha}\|\na \rho\|_{L^q}^2\left(\|\na u\|_{L^2}^2+\|\na w\|_{L^2}^2\right)\nonumber\\
&\leq \frac{\mu_1}{2^{\alpha+3}}\br^\alpha\|\na u_t\|_{L^2}^2+\frac{\mu_2}{2^{\alpha+3}}\br^\alpha\|\na w_t\|_{L^2}^2+C\br^{2-\alpha}\left(\|\na u\|_{L^2}^2+\|\na w\|_{L^2}^2\right)\label{3.35}
\end{align}
due to $b\leq 2$. Using $\beta< \alpha$, we can estimate $J_7-J_{11}$ as follows:
\begin{align}
\sum_{i=7}^{11}J_i&\leq C\br^{\alpha-1}\int|\rho\dg u+\na\rho\cdot u|\left(|\na u||\na u_t|+|\na w||\na w_t|\right)\md x\nonumber\\
&\leq C\br^\alpha\|\na u\|_{L^4}^2\|\na u_t\|_{L^2}+C\br^{\alpha-1}\|\na\rho\|_{L^q}\|u\|_{\frac{3q}{q-3}}\|\na u\|_{L^6}\|\na u_t\|_{L^2}\nonumber\\
&\quad+C\br^\alpha\|\na u\|_{L^4}\|\na w\|_{L^4}\|\na w_t\|_{L^2}+C\br^{\alpha-1}\|\na\rho\|_{L^q}\|u\|_{\frac{3q}{q-3}}\|\na w\|_{L^6}\|\na w_t\|_{L^2}\nonumber\\
&\leq \frac{\mu_1}{2^{\alpha+3}}\br^\alpha\|\na u_t\|_{L^2}^2+\frac{\mu_2}{2^{\alpha+3}}\br^\alpha\|\na w_t\|_{L^2}^2\nonumber\\
&\quad+C\br^{\alpha}\|\na u\|_{L^2}\|\na u\|_{L^6}^3+C\br^{\alpha}\|\na w\|_{L^2}\|\na w\|_{L^6}^3\nonumber\\
&\quad+C\br^{\alpha-2+b}\|\na u\|_{L^2}^{\frac{3(q-2)}{q}}\|\na u\|_{L^6}^{2+\frac{6-q}{q}}+C\br^{\alpha-2+b}\|\na u\|_{L^2}^{\frac{3(q-2)}{q}}\|\na u\|_{L^6}^{\frac{6-q}{q}}\|\na w\|_{L^6}^2.
\label{3.36}
\end{align}
Similarly,
\begin{align}
\int\bx\left(\rho^{\beta}\right)_t|w||w_t|\md x&\leq\int\bx\rho^{\beta-1}|\rho\dg u+\na\rho\cdot u||w||w_t|\md x\nonumber\\
&\quad\leq C\bx\br^{\beta}\|\na u\|_{L^2}\|w\|_{L^3}\|w_t\|_{L^6}\nonumber\\
&\qquad+C\bx\br^{\beta-1}\|\na \rho\|_{L^3}\|u\|_{L^6}\|w\|_{L^3}\|w_t\|_{L^6}\nonumber\\
&\quad\leq \frac{\mu_1}{2^{\alpha+3}}\br^\alpha\|\na w_t\|_{L^2}^2+C\bx^2\br^{2\beta-\alpha}\left(\|\na u\|_{L^2}^2+\|\na w\|_{L^2}^2\right),
\label{3.36'}
\end{align}
and $J_{13}$ is bounded by
\begin{align}
J_{13}&\leq C\br^{\gamma-1}\int|\rho\dg u+\na\rho\cdot u||\dg u_t|\md x\nonumber\\
&\leq C\br^\gamma\|\dg u\|_{L^2}\|\dg u_t\|_{L^2}+C\br^{\gamma-1}\|\na\rho\|_{L^q}\|u\|_{L^{\frac{2q}{q-2}}}\|\dg u_t\|_{L^2}\nonumber\\
&\leq \frac{\mu_1}{2^{\alpha+3}}\br^\alpha\|\na u_t\|_{L^2}^2+C\br^{2\gamma-\alpha}\|\na u\|_{L^2}^2+C\br^{2\gamma-2-\alpha}\|\na\rho\|_{L^q}^2\|u\|_{L^2}^{\frac{2(q-3)}{q}}\|\na u\|_{L^2}^{\frac{6}{q}}\nonumber\\
&\leq \frac{\mu_1}{2^{\alpha+3}}\br^\alpha\|\na u_t\|_{L^2}^2+C\br^{2\gamma-\alpha}\|\na u\|_{L^2}^2+C\br^{2\gamma-2-\alpha}\|\na\rho\|_{L^q}^2.
\label{3.37}
\end{align}

putting (\ref{3.34})-(\ref{3.37}) into (\ref{3.33}) and integrating the result inequality yields that
\begin{align}
&\frac{1}{2}\sup_{t\in[0,T]}\left(\|\sqrt{\rho}u_t\|_{L^2}^2+\|\sqrt{\rho}w_t\|_{L^2}^2\right)\md x+\frac{\br^\alpha}{2^\alpha}\int_0^T\left(\mu_1\|\na u_t\|_{L^2}^2+\mu_2\|\na w_t\|_{L^2}^2\right)\md t\nonumber\\
&\quad+\frac{\br^\alpha}{2^\alpha}\int_0^T\left((\mu_1+\lambda_1)\|\dg u_t\|_{L^2}^2+(\mu_2+\lambda_2)\|\dg w_t\|_{L^2}^2\right)\md t\nonumber\\
&\leq \frac{1}{2}\sup_{t\in[0,T]}\left(\|\sqrt{\rho}u_t\|_{L^2}^2+\|\sqrt{\rho}w_t\|_{L^2}^2\right)\md x\md t+\int_0^T\int\rho^\alpha\left(2\mu_1|d (u_t)|^2+\mu_2|\na w_t|^2\right)\md x\md t\nonumber\\
&\quad+\int_0^T\int\rho^\alpha\left(\lambda_1(\dg u_t)^2+(\mu_2+\lambda_2)(\dg w_t)^2\right)\md x\md t\nonumber\\
&\quad+\int_0^T\int 2\bx \rho^\beta |d(u_t)|^2\md x\md t+\bx\int_0^T\int\rho^\beta|\cl u_t-2w_t|^2\md x\md t\nonumber\\
&\leq \frac{1}{2}\left(\|\sqrt{\rho_0}u_{0t}\|_{L^2}^2+\|\sqrt{\rho_0}w_{0t}\|_{L^2}^2\right)+\frac{\br^\alpha}{2^{\alpha+1}}\int_0^T\left(\mu_1\|\na u_t\|_{L^2}^2+\mu_2\|\na w_t\|_{L^2}^2\right)\md t\nonumber\\
&\quad+C\br^{2\gamma-2-\alpha}\int_0^T\|\na\rho\|_{L^q}^2\md t+\br^\alpha\left(\int_0^T\|\na u\|_{L^2}^2\md t\right)^{\frac{1}{2}}+\br^\alpha\left(\int_0^T\|\na w\|_{L^2}^2\md t\right)^{\frac{1}{2}}\nonumber\\
&\quad+C\left(\br^{2-\alpha}+\br^{2\gamma-\alpha}+\br^{2\beta-\alpha}\right)\int_0^T\left(\|\na u\|_{L^2}^2+\|\na w\|_{L^2}^2\right)\md t\nonumber\\
&\leq \frac{1}{2}\left(\|\sqrt{\rho_0}u_{0t}\|_{L^2}^2+\|\sqrt{\rho_0}w_{0t}\|_{L^2}^2\right)+\frac{\br^\alpha}{2^{\alpha+1}}\int_0^T\left(\mu_1\|\na u_t\|_{L^2}^2+\mu_2\|\na w_t\|_{L^2}^2\right)\md t\nonumber\\
&\quad+C\left(\br^{\gamma-2+b}+\br^{\frac{\alpha+1}{2}}+\br^{3-2\alpha}+\br^{2\gamma+1-2\alpha}+\br^{2\beta+1-2\alpha}\right)\nonumber\\
&\leq \frac{1}{2}\left(\|\sqrt{\rho_0}u_{0t}\|_{L^2}^2+\|\sqrt{\rho_0}w_{0t}\|_{L^2}^2\right)+\frac{\br^\alpha}{2^{\alpha+1}}\int_0^T\left(\mu_1\|\na u_t\|_{L^2}^2+\mu_2\|\na w_t\|_{L^2}^2\right)\md t\nonumber\\
&\quad+C\left(\br^{\gamma-2+b}+\br^{\frac{\alpha+1}{2}}+\br^{2\gamma+1-2\alpha}\right),\label{3.39}
\end{align}
where we have used
\begin{equation*}
3-2\alpha<\frac{\alpha+1}{2},\quad 2\beta+1-2\alpha<\frac{\alpha+1}{2}
\end{equation*}
due to $\alpha>1$ and $\beta<\alpha$.

Define
\begin{align}
\rho_0^{1/2}u_{0t}&\triangleq\rho_0^{-1/2}\left(-\rho_0u_0\cdot\na u_0-\na P(\rho_0)+\dg\left(2(\mu_1\rho_0^\alpha+ \bx\rho_0^\beta)d(u_0)\right)\right.\nonumber\\
&\quad+\left.\nabla\left(\lambda_1\rho_0^\alpha\dg u_0\right)-\cl\left(\bx\rho_0^\beta\cl u_0\right)+2\bx\cl\left(\rho_0^\beta w_0\right)\right),\label{3.40}\\
\rho_0^{1/2}w_{0t}&\triangleq\rho_0^{-1/2}\left(-\rho_0u_0\cdot\na w_0+\dg\left(\mu_2\rho_0^\alpha \nabla u_0\right)\right.\nonumber\\
&\quad\left.+\nabla\left(\mu_2+\lambda_2)\rho_0^\alpha\dg w_0\right)+2\bx\rho_0^\beta\cl u_0-4\bx\rho_0^\beta w_0\right).\label{3.41}
\end{align}
Then it follows from (\ref{1.6}) that
\begin{align}
&\|\sqrt{\rho_0}u_{0t}\|_{L^2}+\|\sqrt{\rho_0}w_{0t}\|_{L^2}\nonumber\\
&\quad\leq  C\br^{\frac{1}{2}}\left(\|\na u_0\|_{L^2}+\|\na w_0\|_{L^2}\right)\cdot\left(\|\na u_0\|_{L^2}^{\frac{1}{2}}\|\na u_0\|_{L^6}^{\frac{1}{2}}+\|\na w_0\|_{L^2}^{\frac{1}{2}}\|\na w_0\|_{L^6}^{\frac{1}{2}}\right))\nonumber\\
&\qquad+C\br^{\gamma-\frac{3}{2}}\|\na\rho_0\|_{L^2}+\br^{\alpha-\frac{1}{2}}\|\na^2 u_0\|_{L^2}+\br^{\alpha-\frac{1}{2}}\|\na^2 w_0\|_{L^2}\nonumber\\
&\qquad+C\br^{\alpha-\frac{3}{2}}\|\na \rho_0\|_{L^q}\|\na u_0\|_{\frac{2q}{q-2}}+C\br^{\alpha-\frac{3}{2}}\|\na \rho_0\|_{L^q}\|\na w_0\|_{\frac{2q}{q-2}}\nonumber\\
&\qquad+C\br^{\beta-\frac{3}{2}}\|\na \rho_0\|_{L^3}\|w_0\|_{L^6}+C\br^{\beta-\frac{1}{2}}\left(\|\na u_0\|_{L^2}+\|w_0\|_{H^1}\right)\nonumber\\
&\quad\leq C\left(\br^{\frac{1}{2}}+\br^{\gamma-\frac{3}{2}}+\br^{\alpha-\frac{1}{2}}+C\br^{\beta-\frac{3}{2}+\frac{b}{2}}+C\br^{\beta-\frac{1}{2}}\right)\leq C_4\br^{\alpha-\frac{1}{2}},\label{3.42}
\end{align}
due to $\alpha>1,\,\gamma-1< \alpha,\,b<2$ and $\beta<\alpha$.

Combining (\ref{3.39}) and (\ref{3.42}) yields that
\begin{align}
&\sup_{t\in[0,T]}\left(\|\sqrt{\rho}u_t\|_{L^2}^2+\|\sqrt{\rho}w_t\|_{L^2}^2\right)\md x+\frac{\br^\alpha}{2^{\alpha}}\int_0^T\left(\mu_1\|\na u_t\|_{L^2}^2+\mu_2\|\na w_t\|_{L^2}^2\right)\md t\nonumber\\
&\quad\leq N_4\br^{2\alpha-1}+C_5\left(\br^{\gamma}+\br^{\frac{\alpha+1}{2}}+\br^{2\gamma+1-2\alpha}\right)\leq  2N_4\br^{2\alpha-1},\label{3.43}
\end{align}
after defining
\begin{align}N_4=N_4&\left(A,\,\mu_1,\,\lambda_1,\,\mu_2,\,\lambda_2,\,\bx,\,\alpha,\,\beta,\,\gamma,\,\|G(\rho_0)\|_{L^1},\,\|\na \rho_0\|_{L^2},\,\|\na \rho_0\|_{L^q},\right.\nonumber\\
&\,\,\|\na u_0\|_{H^1},\,\|w_0\|_{H^2}\Big{)}\triangleq C_4^2\label{3.44}\end{align}
and choosing
\begin{align*}\br\geq M_3\triangleq\max&\left\{M_2,\,\left(\frac{3C_5}{N_4}\right)^{\frac{1}{2\alpha-\gamma-1}}\left(\frac{3C_5}{N_4}\right)^{\frac{2}{3\alpha-3}},\,\,\left(\frac{3C_5}{N_4}\right)^{\frac{1}{2(2\alpha-\gamma-1)}}\right\}.\end{align*}
\hfill $\square$

Now, the next lemma is concentrated on the estimates of density.
\begin{lem}\label{L5}
Assume that
\begin{equation}\label{3.45}\left\{
\begin{aligned}
&3-\gamma\leq b<2,\quad 1<\gamma<3,\\
&b=0,\quad\gamma\geq 3,
\end{aligned}
\right.
\end{equation}
(\ref{1.7}) and (\ref{1.8}) hold. Then
$$\varepsilon_{\rho,2}(T)\leq N_1\br^b,\quad \varepsilon_{\rho,3}(T)\leq N_2\br^b,$$
provided $\br>M_4(A,\alpha,\,\beta,\,\gamma,$ $\mu_1,\,\mu_2,\,\lambda_1,\,\lambda_2$, $\bx$, $G(\rho_0)$, $\|\nabla\rho_0\|_{L^2}$, $\|\nabla\rho_0\|_{L^q}$, $\|u_0\|_{H^2}$, $\|w_0\|_{H^2})$, Here $N_1$ and $N_2$ are defined in (\ref{3.60}) and (\ref{3.65}), respectively.
\end{lem}
\pf Applying $\na$ to $(\ref{1.1})_1$, using $(\ref{3.16})_1$, we can obtain that
\begin{equation}\label{3.46}
\na\rho_t+u\cdot\na^2\rho+\na u\cdot\na\rho+\na\rho\dg u+\frac{1}{2\mu_1+\lambda_1}\rho(\na F_1+\na \mathcal{P})=0.
\end{equation}
Multiplying the above equation by $|\na\rho|^{q-2}\na\rho$, integrating by parts yields that
\begin{equation}\label{3.47}
\frac{\md}{\md t}\|\na\rho\|_{L^q}^2+\br^{\gamma-\alpha}\|\na\rho\|_{L^q}^2\leq C\|\na u\|_{L^\infty}\|\na\rho\|_{L^q}^2+C\br\|\na F_1\|_{L^q}\|\na\rho\|_{L^q}.
\end{equation}
Choosing $\br\geq\left(4C_3\right)^{\frac{1}{\alpha-\beta}}$ in (\ref{3.18}), we can obtain that
\begin{align}
&\|\nabla^2 u\|_{L^q}+\|\nabla^2 w\|_{L^q}\nonumber\\
&\leq C\br^{-\alpha}\left(\|\rho u_t\|_{L^q}+\|\rho w_t\|_{L^q}\right)\nonumber\\
&\quad+C{\bar{\rho}}^{-\alpha+1}\| u\|_{L^q}\|\nabla u\|_{L^\infty}+C{\bar{\rho}}^{-1}\|\nabla \rho\|_{L^q}\|\nabla u\|_{L^\infty}\nonumber\\
&\quad+C{\bar{\rho}}^{-\alpha+\gamma-1}\|\nabla \rho\|_{L^q}+C{\bar{\rho}}^{-\alpha+1}\| u\|_{L^q}\|\nabla w\|_{L^\infty}
+C{\bar{\rho}}^{-1}\|\nabla \rho\|_{L^q}\|\nabla w\|_{L^\infty}\nonumber\\
&\quad+C\br^{(\beta-\alpha)\frac{1}{\theta}}\left(\|\nabla u\|_{L^2}+\|\nabla w\|_{L^2}\right)+C\bx\br^{\beta-\alpha}\|w\|_{L^q}, \label{3.48}
\end{align}
where $\theta$ is defined in (\ref{3.19}). Combing (\ref{3.48}) with (\ref{3.18}), we can obtain that
\begin{align}
\|\na F_1\|_{L^q}
&\leq C\br^{-\alpha}\left(\|\rho u_t\|_{L^q}+\|\rho w_t\|_{L^q}\right)\nonumber\\
&\quad+C{\bar{\rho}}^{-\alpha+1}\| u\|_{L^q}\|\nabla u\|_{L^\infty}+C{\bar{\rho}}^{-1}\|\nabla \rho\|_{L^q}\|\nabla u\|_{L^\infty}\nonumber\\
&\quad+C{\bar{\rho}}^{\beta-2\alpha+\gamma-1}\|\nabla \rho\|_{L^q}+C{\bar{\rho}}^{-\alpha+1}\| u\|_{L^q}\|\nabla w\|_{L^\infty}
+C{\bar{\rho}}^{-1}\|\nabla \rho\|_{L^q}\|\nabla w\|_{L^\infty}\nonumber\\
&\quad+C\br^{(\beta-\alpha)\frac{1}{\theta}}\left(\|\nabla u\|_{L^2}+\|\nabla w\|_{L^2}\right)+C\bx\br^{\beta-\alpha}\|w\|_{L^q}.\label{3.49}
\end{align}
Inserting (\ref{3.49}) into (\ref{3.47}) gives
\begin{align}
&\frac{\md}{\md t}\|\na\rho\|_{L^q}^2+\br^{\gamma-\alpha}\|\na\rho\|_{L^q}^2\nonumber\\
&\quad\leq C\left(\|\nabla u\|_{L^\infty}+\|\nabla w\|_{L^\infty}\right)\|\na\rho\|_{L^q}^2+C\br^{1-\alpha}\left(\|\rho u_t\|_{L^q}+\|\rho w_t\|_{L^q}\right)\|\na\rho\|_{L^q}\nonumber\\
&\qquad+C\br^{2-\alpha}\|u\|_{L^q}\left(\|\nabla u\|_{L^\infty}+\|\nabla w\|_{L^\infty}\right)\|\na\rho\|_{L^q}+C\br^{\gamma-\alpha+\beta-\alpha}\|\na\rho\|_{L^q}^2\nonumber\\
&\qquad+C\br^{1+(\beta-\alpha)\frac{1}{\theta}}\left(\|\nabla u\|_{L^2}+\|\nabla w\|_{L^2}\right)\|\na\rho\|_{L^q}+C\br^{1+\beta-\alpha}\|w\|_{L^q}\|\na\rho\|_{L^q}\nonumber\\
&\quad\leq \frac{1}{4}\br^{\gamma-\alpha}\|\na\rho\|_{L^q}^2+C_6\br^{\gamma-\alpha+\beta-\alpha}\|\na\rho\|_{L^q}^2+C\br^{\alpha-\gamma}\left(\|\na u\|_{L^\infty}^2+\|\na w\|_{L^\infty}^2\right)\|\na\rho\|_{L^q}^2\nonumber\\
&\qquad+C\br^{2-\gamma-\alpha}\left(\|\rho u_t\|_{L^q}^2+\|\rho w_t\|_{L^q}^2\right)+C\br^{4-\alpha-\gamma}\left(\|\na u\|_{L^\infty}^2+\|\na w\|_{L^\infty}^2\right)\nonumber\\
&\qquad+C\br^{2-\gamma+\alpha+(\beta-\alpha)\frac{2}{\theta}}\left(\|\nabla u\|_{L^2}^2+\|\nabla w\|_{L^2}^2\right)+C\bx\br^{2+2\beta-\gamma-\alpha}\|w\|_{L^q}^2.\label{3.50}
\end{align}

It follows from (\ref{3.3}) that
\begin{align}
\int_0^T\|\rho u_t\|_{L^q}^2\md t&\leq \int_0^T\br^{\frac{5q-6}{2q}}\|\sqrt{\rho}u_t\|_{L^2}^{\frac{6-q}{q}}\|\na u_t\|_{L^2}^{\frac{3(q-2)}{q}}\md t\nonumber\\
&\leq C\br^{\frac{5q-6}{2q}}\left( \int_0^T\|\sqrt{\rho}u_t\|_{L^2}^2\md t\right)^{\frac{6-q}{2q}}\left(\int_0^T\|\na u_t\|_{L^2}^{2}\md t\right)^{\frac{3(q-2)}{2q}}\nonumber\\
&\leq C\br^{\frac{5q-6}{2q}+\alpha{\frac{6-q}{2q}}+(\alpha-1){\frac{3(q-2)}{2q}}}=C\br^{\alpha+1},\label{3.51}
\end{align}
Similarly,
\begin{align}
\int_0^T\|\rho w_t\|_{L^q}^2\md t\leq C\br^{\alpha+1}.\label{3.52}
\end{align}
Note that by Sobolev inequality
\begin{equation}\|\na u\|_{L^\infty}\leq C\|\na u\|_{L^2}^{\theta'}\|\na^2 u\|_{L^q}^{1-\theta'},\quad\|\na w\|_{L^\infty}\leq C\|\na w\|_{L^2}^{\theta'}\|\na^2 w\|_{L^q}^{1-\theta'},\label{3.53}\end{equation}
where
\begin{equation}
\theta'=\frac{2(q-3)}{5q-6}.\label{3.54}
\end{equation}
Combing (\ref{3.48}), (\ref{3.51})$-$(\ref{3.53}) and (\ref{3.3}) gives
\begin{align}
\int_0^T\|\na u\|_{L^\infty}^2\md t&\leq \int_0^T\|\na u\|_{L^2}^{2\theta'}\|\na^2 u\|_{L^q}^{2(1-\theta')}\md t\nonumber\\
&\leq C\int_0^T\left(\|\na u\|_{L^2}^2+\br^{-2\alpha}\left(\|\rho u_t\|_{L^q}^2+\|\rho w_t\|_{L^q}^2\right)\right.\nonumber\\
&\quad+{\bar{\rho}}^{\frac{-2\alpha+2}{\theta'}}\|u\|_{L^q}^{\frac{2}{\theta'}}\left(\|\nabla u\|_{L^2}^2+\|\nabla w\|_{L^2}^2\right)\nonumber\\
&\quad+\br^{-\frac{2}{\theta'}}\|\na\rho\|_{L^q}^{\frac{2}{\theta'}}\left(\|\na u\|_{L^2}^2+ \|\na w\|_{L^2}^2\right)+{\bar{\rho}}^{-2\alpha+2\gamma-2}\|\nabla \rho\|_{L^q}^2\nonumber\\
&\quad+\left.\br^{\frac{2\beta-2\alpha }{\theta'}}\left(\|\nabla u\|_{L^2}^2+\|\nabla w\|_{L^2}^2\right)+\br^{2\beta-2\alpha}\|w\|_{L^q}^2\right)\md t\nonumber\\
&\leq C\br^{1-\alpha}+c\br^{\gamma-\alpha-2+b}+\br^{2\beta-2\alpha}\int_0^T\left(\|w\|_{L^2}^2+\|\na w\|_{L^2}^2\right)\md t\nonumber\\
&\leq C\br^{1-\alpha}+C\br^{\gamma-\alpha-2+b}+\br^{1+\beta-2\alpha}\leq C\br^{\gamma-\alpha},\label{3.55}
\end{align}
due to $\alpha>1,\,\gamma>1,\,b>2$ and $\beta< \alpha$.

Similarly
\begin{equation}
\int_0^T\|\na w\|_{L^\infty}^2\md t\leq C\br^{\gamma-\alpha}.\label{3.56}\end{equation}
We also have the following inequalities:
\begin{align}
\br^{2-\gamma+\alpha+(\beta-\alpha)\frac{2}{\theta}}\int_0^T\left(\|\nabla u\|_{L^2}^2+\|\nabla w\|_{L^2}^2\right)\md t\leq C\br^{3-\gamma+(\beta-\alpha)\frac{2}{\theta}}\leq C\br^{3-\gamma},\label{3.57}\end{align}
\begin{align}
C\bx\br^{2+2\beta-\gamma-\alpha}\int_0^T\|w\|_{L^q}^2\md t&\leq C\bx\br^{2+2\beta-\gamma-\alpha}\int_0^T\left(\|w\|_{L^2}^2+\|\na w\|_{L^2}^2\right)\md t\nonumber\\
&\leq C\br^{3-\gamma+\beta-\alpha}\leq C\br^{3-\gamma},\label{3.58}
\end{align}
due to $\beta<\alpha$. Finally, choosing
$$\br\geq\max\left\{M_3,\,\left(4C_3\right)^{\frac{1}{\alpha-\beta}},\,\left(4C_6\right)^{\frac{1}{\alpha-\beta}}\right\},$$
applying Gr$\rm{\ddot{o}}$nwall inequality, and using(\ref{3.50})$-$(\ref{3.52}), (\ref{3.55})$-$(\ref{3.58}), we have
\begin{align}
&\sup_{t\in [0,T]}\|\na\rho\|_{L^q}^2+\frac{1}{2}\br^{\gamma-\alpha}\int_0^T\|\na\rho\|_{L^q}^2\md t\nonumber\\
&\quad\leq C\left(\|\na \rho_0\|_{L^q}^2+\br^{3-\gamma}+\br^{4-2\alpha}+\br^{3-\gamma-\alpha+\beta}\right)\nonumber\\
&\quad\leq C_7\left(\|\na \rho_0\|_{L^q}^2+\br^{3-\gamma}\right)\leq N_1\br^b,\label{3.59}
\end{align}
due to $\alpha>\frac{\gamma+1}{2},\,\beta<\alpha$ where $b$ satisfies (\ref{3.45}) and
\begin{equation}\label{3.60}
N_1\triangleq C_7\left(\|\na \rho_0\|_{L^q}^2+1\right).
\end{equation}
Similarly, Choosing $\br\geq\left(4C_2\right)^{\frac{1}{\alpha-\beta}}$ in (\ref{3.17}),
\begin{align}
&\|\nabla^2 u\|_{L^2}+\|\nabla^2 w\|_{L^2}\nonumber\\
&\leq C\br^{-\alpha+\frac{1}{2}}\left(\|\sqrt{\rho}u_t\|_{L^2}+\|\sqrt{\rho}w_t\|_{L^2}\right)+C{\bar{\rho}}^{-2\alpha+2}\left(\|\nabla u\|_{L^2}^3+\|\nabla w\|_{L^2}^3\right)\nonumber\\
&\quad+C\br^{-\alpha+\gamma-1}\|\na\rho\|_{L^2}+C\br^{-\frac{q}{q-3}}\|\na\rho\|_{L^q}^{\frac{q}{q-3}}\left(\|\na u\|_{L^2}+\|\na w\|_{L^2}\right)\nonumber\\
&\quad+C\bx\br^{\beta-\alpha}(\|\nabla u\|_{L^2}+\|\nabla w\|_{L^2}+\| w\|_{L^2}).\label{3.61}
\end{align}
Collecting $(\ref{3.3})$ and (\ref{3.61}) gives
\begin{align}
\|\na F_1\|_{L^2}&\leq C\br^{-\alpha+\frac{1}{2}}\left(\|\sqrt{\rho}u_t\|_{L^2}+\|\sqrt{\rho}w_t\|_{L^2}\right)+C{\bar{\rho}}^{-2\alpha+2}\left(\|\nabla u\|_{L^2}^3+\|\nabla w\|_{L^2}^3\right)\nonumber\\
&\quad+C\br^{\beta-2\alpha+\gamma-1}\|\na\rho\|_{L^2}+c\br^{-\frac{q}{q-3}}\|\na\rho\|_{L^q}^{\frac{q}{q-3}}\left(\|\na u\|_{L^2}+\|\na w\|_{L^2}\right)\nonumber\\
&\quad+C\bx\br^{\beta-\alpha}(\|\nabla u\|_{L^2}+\|\nabla w\|_{L^2}+\|w\|_{L^2})\nonumber\\
&\leq C\br^{-\alpha+\frac{1}{2}}\left(\|\sqrt{\rho}u_t\|_{L^2}+\|\sqrt{\rho}w_t\|_{L^2}\right)+C{\bar{\rho}}^{-2\alpha+2}\left(\|\nabla u\|_{L^2}+\|\nabla w\|_{L^2}\right)\nonumber\\
&\quad+C\br^{\beta-2\alpha+\gamma-1}\|\na\rho\|_{L^2}+C\br^{\left(\frac{b}{2}-1\right){\frac{q}{q-3}}}\left(\|\na u\|_{L^2}+\|\na w\|_{L^2}\right)\nonumber\\
&\quad+C\bx\br^{\beta-\alpha}(\|\nabla u\|_{L^2}+\|\nabla w\|_{L^2}+\| w\|_{L^2}).\label{3.62}
\end{align}
Using (\ref{3.54}), we have
\begin{align}
&\frac{\md}{\md t}\|\na\rho\|_{L^2}^2+\br^{\gamma-\alpha}\|\na\rho\|_{L^2}^2\nonumber\\
&\quad\leq C\|\na u\|_{L^\infty}\|\na\rho\|_{L^2}^2+C\br\|\na F_1\|_{L^2}\|\na\rho\|_{L^2}\nonumber\\
&\quad\leq \frac{1}{4}\br^{\gamma-\alpha}\|\na\rho\|_{L^2}^2+C\br^{\alpha-\gamma}\|\na u\|_{L^\infty}^2\|\na\rho\|_{L^2}^2+C\br^{2-\gamma+\alpha}\|\na F_1\|_{L^2}^2\nonumber\\
&\quad\leq \frac{1}{4}\br^{\gamma-\alpha}\|\na\rho\|_{L^2}^2+C_{8}\br^{\gamma-3\alpha+2\beta}\|\na\rho\|_{L^2}^2+C\br^{\alpha-\gamma}\|\na u\|_{L^\infty}^2\|\na\rho\|_{L^2}^2\nonumber\\
&\qquad+C\br^{3-\gamma-\alpha}\left(\|\sqrt{\rho} u_t\|_{L^2}^2+\|\sqrt{\rho} w_t\|_{L^2}^2\right)+C\br^{6-3\alpha-\gamma}\left(\|\nabla u\|_{L^2}^2+\|\nabla w\|_{L^2}^2\right)\nonumber\\
&\qquad+C\bx\br^{2-\gamma+\beta-\alpha}\br^\beta\|w\|_{L^2}^2.\label{3.63}
\end{align}
Choosing
$$\br\geq\max\left\{M_3,\,\left(4C_2\right)^{\frac{1}{\alpha-\beta}},\,\left(4C_{8}\right)^{\frac{1}{\alpha-\beta}}\right\},$$
combining (\ref{3.61}), (\ref{3.62}) and (\ref{3.3}), and applying Gr${\rm\ddot{o}}$nwall inequality to (\ref{3.63}) yields that
\begin{align}
&\sup_{t\in [0,T]}\|\na\rho\|_{L^2}^2+\frac{1}{2}\br^{\gamma-\alpha}\int_0^T\|\na\rho\|_{L^2}^2\md t\nonumber\\
&\quad\leq C\left(\|\na \rho_0\|_{L^2}^2+\br^{3-\gamma}+\br^{3-\gamma+4-4\alpha}+\br^{3-\gamma+\beta-\alpha}\right)\nonumber\\
&\quad\leq C_{9}\left(\|\na \rho_0\|_{L^2}^2+\br^{3-\gamma}\right)\leq N_2\br^b\label{3.64}
\end{align}
due to $\alpha>1$ and $\beta<\alpha$, where
\begin{equation}\label{3.65}
N_2\triangleq C_{9}\left(\|\na \rho_0\|_{L^2}^2+1\right).\end{equation}

Thus letting
$$M_4\triangleq \max\left\{M_3,\,\left(4C_2\right)^{\frac{1}{\alpha-\beta}},\,\left(4C_3\right)^{\frac{1}{\alpha-\beta}},\,\left(4C_6\right)^{\frac{1}{\alpha-\beta}},\,\left(4C_{8}\right)^{\frac{1}{\alpha-\beta}}\right\}$$
and choosing $\br> M_4$, (\ref{3.59}) and (\ref{3.64}) are satisfied. The proof of Lemma \ref{L5} is completed.\hfill $\square$

\textit{Proof of Proposition} 3.1. It only suffices to prove that
\begin{align}
\varepsilon_{\rho,1}(T)&=\sup_{t\in [0,T]}\|\rho-\br\|_{L^\infty}\nonumber\\
&\leq \sup_{t\in [0,T]}\|\rho-\br\|_{L^2}^{\frac{2(q-3)}{5q-6}}\|\na \rho\|_{L^q}^{\frac{3q}{5q-6}}\nonumber\\
&\leq C\br^{{\frac{2(q-3)}{5q-6}}{\frac{3-\gamma}{2}}+{\frac{3q}{5q-6}}\frac{b}{2}}\nonumber\\
&\leq C_{10}\br^{{\frac{2(q-3)}{5q-6}}+{\frac{3q}{5q-6}}\frac{b}{2}}\leq \frac{\br}{4},\label{3.66}
\end{align}
provided that
\begin{equation}
\br\geq \Lambda_0\triangleq\max\left\{M_4,\,\left(4C_{10}\right)^{\frac{2(5q-6)}{3q(2-b)}}\right\}\label{3.67}
\end{equation}
and $$b<2.$$

Finally, choosing $b$ as in (\ref{3.2}). If
$$\alpha>1,\,\,\gamma>1,\,\,2\alpha>\gamma+1,$$
and
$$\gamma-1\leq \alpha,$$
$b$ satisfies (\ref{3.20}) and (\ref{3.45}) automatically. Hence, (\ref{3.4}) follows from Lemmas \ref{L3}$-$\ref{L5}, and the proof of Proposition \ref{p} is finished. \hfill $\square$
\section{Proof of Theorem \ref{th1}}
By Lemma \ref{lem2.1}, there exists a positive time $T_{*}$ such that the Cauchy problem $(\ref{1.1})-(\ref{1.5})$ has a strong solution $(\rho, u, w)$ in $\mathbb{R}^3\times (0, T_{*}]$. It follows from (\ref{1.6})) that there exists a $T_1 \in (0,T_{*}] $ such that (\ref{3.3}) holds for $T = T_1$. So, to prove Theorem \ref{th1}, it suffices to prove that the local solutions can be extended to be a global one. To this end, we assume from now on that $\br\geq \Lambda_0$ holds with $\Lambda_0$ being defined in Proposition {\ref{p}.

Set
\begin{equation}\label{4.1}
T^*=\sup\left\{T|(\rho,u,w) {\rm \,\,is\,\, a\,\, srtong\,\, solution\,\, on\,\,} \mathbb{R}^3\times[0,T]\,\,and\,\,(\ref{3.3})\,\,holds\right\}.
\end{equation}
Then $T^*\geq T_1>0$. We claim that
\begin{equation}\label{4.2}
T^*=\infty.
\end{equation}
In fact, if $T^*<\infty$. By virtue of Proposition \ref{p}, it holds that $(\rho,u,w)|_{t=T^*}$ satisfies (\ref{1.6}) and (\ref{3.3}) holds for $T=T^*$ . Thus, Lemma \ref{lem2.1} implies that there exists $T^{**}>T^*$ such that $(\rho,u,w)$ can be extended to be a strong strong solution to(\ref{1.1})$-$(\ref{1.5}) in $\mathbb{R}^3\times (0, T_{**}]$, which contradicts (\ref{4.1}). Hence (\ref{4.2}) holds.

\section{Proof of Remark \ref{re5}}
The key point of the proof of Remark \ref{re5} is the $L^p$ estimates of elliptic system. For convenience, we set $\lambda_1(\rho)=\lambda_1$ and $\lambda_2(\rho)=\lambda_2$ to be constants and  satisfy
\begin{align*}
\mu_1>0,\,\,2\mu_1+3\lambda_1\rho^{-\alpha}\geq 0,\\
\mu_2>0,\,\,2\mu_2+3\lambda_2\rho^{-\alpha}\geq 0.
\end{align*}

Next, rewrite $(\ref{1.1})_2$ and $(\ref{1.1})_3$ as
\begin{equation}\label{5.1}
\left\{\begin{aligned}
&\nabla\left(\left(2\mu_1+\lambda_1\rho^{-\alpha}\right)\dg u\right)-\mu_1\cl\cl u=\nabla \mathcal{P}+G_1,\\
&\nabla\left(\left(2\mu_2+\lambda_2\rho^{-\alpha}\right)\dg w\right)-\mu_2\cl\cl w=G_2,
\end{aligned}\right.
\end{equation}
where
\be\label{5.2}
\nabla \mathcal{P}=\left\{
\begin{aligned}
&\frac{A\gamma}{\gamma-\alpha}\nabla\rho^{\gamma-\alpha},\quad&\gamma\neq\alpha,\\
&A\gamma\nabla{\rm ln}\rho,\quad&\gamma=\alpha,
\end{aligned}
\right.
\ee
and
\begin{align}
G_1\triangleq&\rho^{1-\alpha} \dot u-2\mu_1\alpha\rho^{-1}\nabla\rho\otimes d(u)-\rho^{-\alpha}\dg (2\bx\rho^\beta d(u))+\rho^{-\alpha}\cl(\bx\rho^\beta\cl u) \nonumber\\
&+A\gamma\rho^{\gamma-\alpha-1}\nabla\rho-2\bx\rho^{-\alpha}\cl\left(\rho^{\beta}w\right)-\lambda_1\alpha\rho^{-\alpha-1}\nabla\rho\dg u,\label{5.3}\\
G_2\triangleq&\rho^{1-\alpha}\dot{w}-\mu_2\alpha\rho^{-1}\nabla\rho\cdot\nabla w-\mu_2\alpha\rho^{-1}\nabla\rho\dg w\nonumber\\
&-2\bar{\xi}\rho^{\beta-\alpha}\cl u+4\bar{\xi}\rho^{\beta-\alpha} w-\lambda_2\alpha\rho^{-\alpha-1}\nabla\rho\dg w.\label{5.4}
\end{align}

Similar with above proofs, we can construct the elliptic equations
\begin{equation}\label{5.5}\begin{aligned}
\Delta K_1=\dg G_1,\quad \Delta V_1=\cl G_1,\\
\Delta K_2=\dg G_2,\quad \Delta V_2=\cl G_2,
\end{aligned}\end{equation}
where $K_1,\,K_2,\,V_1,\,V_2$ satisfying
\begin{equation}\label{5.6}\left\{
\begin{aligned}
&K_1=\left(2\mu_1+\lambda_1\rho^{-\alpha}\right)\dg u-\mathcal{P}(\rho)+\mathcal{P}({\bar{\rho}}),\quad &V_1=\mu_1\cl u,\\
&K_2=\left(2\mu_2+\lambda_2\rho^{-\alpha}\right)\dg w,\quad &V_2=\mu_2\cl w.
\end{aligned}\right.\end{equation}

Applying the standard $L^p$-estimates to the elliptic system (\ref{5.5}), we can prove that (\ref{3.17}) and (\ref{3.18}) hold true also in the similar method in section 3 due to (\ref{3.7}). Then we can prove the Remark \ref{re5} step by step similar as in Section 3 and Section 4.
\\\\
{\bf Conflict of interest}\\
This work does not have any conflicts of interest.
\\{\bf Data availability}
\\
Data sharing not applicable to this article as no datasets were generated or analysed during the current study.

\end{document}